\documentclass[a4paper]{article}

\usepackage[cyr]{aeguill}
\usepackage[english]{babel}
\usepackage[utf8]{inputenc}
\usepackage[T1]{fontenc}

\usepackage[a4paper,top=3cm,bottom=2cm,left=3cm,right=3cm,marginparwidth=1.75cm]{geometry}

\usepackage{amsmath}
\usepackage{graphicx}
\graphicspath{ {./images/} }
\usepackage[colorinlistoftodos]{todonotes}
\usepackage[colorlinks=true, allcolors=blue]{hyperref}
\usepackage{amssymb}
\usepackage{amsthm}
\usepackage{bbold}
\usepackage[all]{xy}
\usepackage{dirtytalk}
\usepackage[titletoc,title]{appendix}
\usepackage{ stmaryrd }

\newcommand{\vertiii}[1]{{\left\vert\kern-0.25ex\left\vert\kern-0.25ex\left\vert #1 
    \right\vert\kern-0.25ex\right\vert\kern-0.25ex\right\vert}}

\newtheorem{theorem}{Theorem}[section]
\newtheorem{corollary}[theorem]{Corollary}
\newtheorem{lemma}[theorem]{Lemma}
\theoremstyle{definition}
\newtheorem{definition}{Definition}[section]
\newtheorem{example}{Example}[section]
\newtheorem{proposition}[theorem]{Proposition}
\newtheorem{remark}[example]{Remark}
\newtheorem{notations}[example]{Notations}

\title{Fourier decay of equilibrium states for bunched attractors}
\author{Gaétan Leclerc}
\date{ }

\begin{document}

\maketitle

\begin{abstract}

Let $M$ be a closed manifold, and let $f:M \rightarrow M$ be a $C^{2+\alpha}$ Axiom A diffeomorphism. Suppose that $f$ has an attractor $\Omega$ with codimension 1 stable lamination. Under a generic nonlinearity condition and a suitable bunching condition, we prove polynomial Fourier decay in the unstable direction for a large class of invariant measures on $\Omega$. Our result applies in particular for the measure of maximal entropy. We construct in the appendix an explicit solenoid that satisfies the nonlinearity and bunching assumption.
\end{abstract}

\section{Introduction}

\subsection{On fractals and chaotic dynamical systems}

It is known since the work of Fatou and Julia that fractals appears as natural invariant sets in various dynamical systems, for example as attractors or repellers. The geometric properties of such invariant subsets depends intimately with the nature of the map $f$, and it is tempting to study the way they relate to each other. One way to relate the geometry of invariant subsets $K$ to the nature of the dynamics is by the mean of invariant measures with support $K$. In the context of hyperbolic dynamics, a natural choice is given by the so-called equilibrium states, a family of probability measures that encompasses the measure of maximal entropy or the SRB measure. \\

If we fix some probability measure with support $K \subset \mathbb{R}^d$, the Fourier transform being a linear bijection $\mathcal{S}'(\mathbb{R}^d) \rightarrow \mathcal{S}'(\mathbb{R}^d)$, the information contained in $\mu$ are the same that the one that are contained in $$\widehat{\mu}(\xi) := \int_K e^{i \xi \cdot x} d\mu(x).$$ 
In a sense, studying $\mu$ or studying $\widehat{\mu}$ should brings the same quantities of information about the dynamics and the geometry of the support. Yet for equilibrium states, even though much is known on the measure in its spatial representation, we have a very poor understanding of its Fourier transform in the general setting. One of the most basic question one can ask about $\widehat{\mu}$ is its behavior when $\xi$ goes to infinity. Does $\widehat{\mu}(\xi) \longrightarrow 0$, and if yes, at which rate ? From the geometric point of view, this question is linked to the notion of Fourier dimension of a set. Recall that for a set $K \subset \mathbb{R}^d$, we call the Fourier dimension of $K$ the quantity $$ \dim_F(K) := \sup \left\{ \alpha \in [0,d] \ | \ \exists \mu \in \mathcal{P}(K) , \exists C>0 , \forall \xi \in \mathbb{R}^d, \  |\widehat{\mu}(\xi)| \leq C(1+|\xi|)^{-\alpha/2}  \right\} $$

where $\mathcal{P}(K)$ denotes the set of probability measures supported on $K$. The Fourier dimension of a set can be thought as a way to quantify additive chaos in the set $K$. It is always less than $\dim_H(E)$, the Hausdorff dimension of $K$. \\

In a probabilistic setting, the work of Salem \cite{Sa51}, Kahane \cite{Ka66}, Bluhm \cite{Bl96} (to cite a few) suggest that sets satisfying $\dim_H K = \dim_F K$ are, in some sense, generic. But except from some specific arithmetic constructions (see \cite{Kau81}, \cite{QR03}, \cite{Ha17} and \cite{FH20}), explicit examples of such sets are hard to find. Even (nontrivial) explicit sets with positive Fourier dimension are difficult to construct. 

\subsection{Recent development and main results}

The first explicit and \say{natural} fractal with positive Fourier dimension (in the sense that it is an invariant set for some chaotic dynamical system) is due to Bourgain and Dyatlov in $\cite{BD17}$. They prove that the limit set of a non-elementary Fuschian Schottky group, seen as a subset of $\mathbb{R}$, has positive Fourier dimension (notice that such sets are Cantor sets). More precisely, Patterson-Sullivan measures exhibit polynomial Fourier decay in this context. \\

This paper introduced a new method to prove Fourier decay of invariant measures based on previous results of Bourgain from additive combinatorics, called "sum-product phenomenons" (see \cite{Bo10}, or  \cite{Gr09} for an accessible introduction). A concrete example of such sum-product result is given in the following theorem.

\begin{theorem}[\cite{BD17}]

For all $\delta>0$, there exists $\varepsilon_1,\varepsilon_2>0$ and $k \in \mathbb{N}$ such that the following holds. Let $\mathcal{Z}$ be a finite set and $\zeta : \mathcal{Z} \rightarrow [1/2,2]$ a map. Let $\eta>1$ be large enough. Assume that for all $\sigma \in \left[ \eta^{-1}, \eta^{-\varepsilon_1} \right]$, $$   \# \left\{ (\mathbf{b}, \mathbf{c}) \in \mathcal{Z}^2 \ , \ |\zeta(\mathbf{b}) - \zeta(\mathbf{c})|< \sigma  \right\} \leq \# \mathcal{Z}^{2} \sigma^\delta . \quad (*) $$
Then
$$ \Big{|} {\# \mathcal{Z}^{-k}} \sum_{(\mathbf{b}_1, \dots , \mathbf{b}_k) \in \mathcal{Z}^k} e^{i \eta \zeta(\mathbf{b}_1) \dots \zeta(\mathbf{b}_k)} \Big{|} \leq \eta^{-\varepsilon_2} .$$

\end{theorem}

Intuitively, the idea is that multiplications may spread the phase so that cancellations happens in the sum. The idea to prove Fourier decay, then, is to use the autosimilarity of equilibrium states to break the Fourier transform into a sum of exponentials on which this sum-product phenomenon applies. One of the difficulty in the strategy being to actually prove the \say{non concentration hypothesis} $(*)$ made on the phase.  \\

The result was generalized in 2019 with the work of Li, Naud and Pan \cite{LNP19} in the context of limits sets for general Kleinian Schottky groups. In this context, the limit set is still a Cantor set, but in the complex plane. \\

The method was generalized further in 2020 by Sahlsten and Stevens \cite{SS20} for very general Cantor sets in the real line. They prove that under a \emph{nonlinearity condition} on the underlying dynamics (in this context, expanding maps on the Cantor set), equilibrium states exhibit polynomial Fourier decay, and hence the Fourier dimension of those Cantor sets is positive. An important remark is that, in the context of the triadic Cantor set $C$, the underlying dynamics is the shift map $x \mapsto 3x\mod 1$, which is linear, and the Fourier dimension of $C$ is zero. \\

The strategy was pushed past the case of Cantor sets by the author in 2021 \cite{Le21} who studied the case of Julia sets for hyperbolic rational maps in the complex plane. The result is the following: when the Julia set is not included in a circle, any equilibrium state associated to $C^1$ potentials enjoys polynomial Fourier decay. The use of Markov partitions, and the control over the dynamics provided by the conformal setting, allowed us to adapt the previous method in this case. Notice that the condition made on the Julia set is a form of nonlinearity condition, as it implies for example that the dynamics is not conjugated (via a Möbius transformation) to $z \mapsto z^d$, which is linear on the circle. \\

In this paper we explore the method in the context of attractors for codimension one Axiom A diffeomorphisms. Our main goal was to investigate the Fourier dimension of such attractors in the nonlinear setting. By the nature of the existing procedure, Fourier decay of equilibrium state couldn't be achieved, but we get a result of polynomial Fourier decay \emph{in the unstable direction} for nonlinear dynamics, as long as a suitable bunching condition is satisfied. \\

Even though this bunching condition seems quite restrictive, we think that the result should hold for attractors of generic codimension one Axiom A diffeomorphisms, as the nonlinearity condition is generic, and as we only use the bunching condition for one argument (to be able to use some Dolgopyat's estimates). On the other hand, the bunching condition has the advantage to allow one to easily construct explicit examples of attractors on which our theorem applies. We prove the genericity of the nonlinearity condition in the appendix A and we give an example of such a construction in the appendix B. \\

It should be stressed that if one succeed to relax the bunching condition, then Fourier decay in the unstable direction for equilibrium state would hold for generic codimension one Anosov diffeomorphisms. In dimension 2, replacing $f$ by $f^{-1}$ would then give Fourier decay in the stable direction, which should be enough to prove that equilibrium states for generic Anosov diffeomorphisms on surfaces exhibit Fourier decay. \\

Let us precise our setting. Let $M$ be a closed manifold, let $f:M \rightarrow M$ be a $C^{2+\alpha}$ Axiom A diffeomorphism (that if, $f$ is $C^2$, and ($d^2f)$ is $\alpha$-Hölder for some $\alpha>0$). Suppose that $f$ has an attractor $\Omega$, and suppose that its stable lamination is codimension 1 (in particular, it is $C^{1+\alpha}$). The nonlinearity condition is defined by putting an independance condition on some Lyapunov exponents.

\begin{definition}

Define $\Omega_{\text{per}} := \{ x \in \Omega \ \text{periodic} \}$. For any $x \in \Omega_{\text{per}}$ with minimal period $n$, define its local unstable Lyapunov exponent $\widehat{\lambda}(x)$ by the formula

$$ \widehat{\lambda}(x) = \frac{1}{n} \sum_{k=0}^{n-1} \ln | \partial_u f( f^k(x) )| $$
where $|\partial_u f|$ is the absolute value of the derivative of $f$ in the unstable direction. \\
We say that $\Omega$ satisfy the nonlinearity condition (NL) if 
$$ \text{dim}_\mathbb{Q} \text{Vect}_\mathbb{Q} \widehat{\lambda}\left( \Omega_\text{per} \right) = \infty .$$

\end{definition}

\begin{definition}
We introduce our bunching condition: see for example \cite{Ha97}, \cite{GRH21}, or \cite{ABV14}, \cite{HP69} for similar conditions in the case of flows. We say that the bunching condition (B) is satisfied if
$$ \forall p \in \Omega, \ |\partial_u f(p)| \cdot \| { {(df)_{p}}_{|E^s_p} } \| < 1 ,$$
where the norm on $(df)_{|E^s}$  is the operator norm induced by the riemannian metric on $M$. It implies that $\Omega$ is a proper attractor. In this case, by continuity of $df$ and compacity of $\Omega$, there exists some $\alpha > 0$ such that $$ \forall p \in \Omega, \ |\partial_u f(p)| \cdot \| { {(df)_{p}}_{|E^s_p} } \|^{1/(1+\alpha)} \leq 1 .$$

In this case, the stable bunching parameter $b^s$ (see $\cite{GRH21}$, and \cite{Ha97} for the pointwise version) satisfies
$$ b^s(p) = 1 + \frac{\ln  \| { {(df)_{p}}_{|E^s_p} } \|^{-1} }{\ln |\partial_u f(p)| } \geq 2+\alpha. $$

This bunching condition implies that the stable lamination is $C^{2+\alpha}$. Also, notice that in dimension 2, the bunching condition becomes $$ \forall p \in \Omega, \ |\det (df)_p| < 1 .$$

That is, we only need the dynamical system $(\Omega,f)$ to be dissipative.

\end{definition}

\begin{theorem}
Let $M$ be a closed manifold. Let $f:M \rightarrow M$ be a $C^{2+\alpha}$ Axiom A diffeomorphism. Suppose that $f$ has an attractor $\Omega$ with codimension 1 stable lamination. Moreover, suppose that $f$ satisfies the nonlinearity condition (NL) and the bunching condition (B). Then the following holds.\\

Let $\mu$ be an equilibrium state for some Hölder potential. For any $C^{\alpha}$ map $\chi:\Omega \rightarrow \mathbb{R}$ and for any $C^{1+\alpha}$ map $\phi:\Omega \rightarrow \mathbb{R}$ such that $\underset{\text{supp} \ \chi}{\inf } |\partial_u \phi|> 0$, we have

$$ \exists C>1, \exists \varepsilon>0, \ \forall \xi \in \mathbb{R}, \ \left|\int_\Omega e^{i \xi \phi} \chi \ d\mu \right| \leq C(1+|\xi|)^{-\varepsilon}.$$

\end{theorem}

\begin{corollary}
Every attractor for nonlinear and dissipative Axiom A diffeomorphism in dimension 2 exhibit polynomial Fourier decay of equilibrium states in the unstable direction. 
\end{corollary}

\subsection{Strategy of the proof}

The goal is to find a way to adapt the existing method for expanding maps in the case of Axiom A diffeomorphisms. For this, we use the standard construction using Markov partitions and then reduce the problem to the one-dimensional case.

\begin{itemize}
    \item In section 2, we collect facts about thermodynamic formalism in the context of Axiom A diffeomorphisms and recall classic constructions. The section 2.7 is devoted to a preliminary regularity result for equilibrium states in our context. In the section 2.8 we state a large deviation result about Birkhoff sums.
    \item In the section 3 we use the large deviations to derive order of magnitude for some dynamically-related quantities.
    \item The proof of Theorem 1.2 begins in the section 4. Using the invariance of the equilibrium state by the dynamics, we first reduce the Fourier transform to an integral on a union of local unstable manifolds, thus reducing the problem to a one dimensional expanding map.
    \item We use a transfer operator to carefully approximate the integral by a sum of exponentials. We then use a generalized version of Theorem 1.1.
    \item We prove the non-concentration hypothesis that is needed to conclude in section 6.
\end{itemize}

\subsection{Acknowledgments}

I would like to thank my PhD supervisor, Frederic Naud, for numerous helpful conversations and for pointing out various helping references, especially on the topic of total non-linearity and Dolgopyat's estimates. I would also like to thank Sébastien Gouëzel and Christophe Dupont for some interesting discussions at the seminar of ergodic theory in Rennes.

\section{Thermodynamic formalism for Axiom A diffeomorphisms}

\subsection{Axiom A diffeomorphisms and basic sets}

We recall standard results about Axiom A diffeomorphisms. An introduction to the topic can be found in \cite{BS02}. A more in-depth study can be found in \cite{KH95}. For an introduction to the thermodynamic formalism of Axiom A diffeomorphisms, we suggest the classic lectures notes of Bowen \cite{Bo75}. 

\begin{definition}

Let $f: M \rightarrow M$ be a diffeomorphism of a closed  $\mathcal{C}^\infty$ riemannian manifold $M$. A compact set $\Lambda \subset M$ is said to be hyperbolic for $f$ if $f(\Lambda) = \Lambda$, and if for each $x \in \Lambda$, the tangent space $T_xM$ can be written as a direct sum $$ T_xM = E^s_x \oplus E^u_x $$
of subspaces such that
\begin{enumerate}
    \item $\forall x\in \Lambda$,  $(df)_x(E^s_x) = E^s_{f(x)}$ and $(df)_x(E^u_x) = E^u_{f(x)} $
    \item $\exists C>0, \ \exists \kappa \in (0,1), \ \forall x \in \Lambda, $
    $$ \forall v \in E^s_x, \ \forall n \geq 0,  \ \|(df^n)_x(v)\| \leq C \kappa^n \| v \| $$
    and $$ \forall v \in E^u_x, \ \forall n \geq 0, \  \|(df^n)_x(v)\| \geq C^{-1} \kappa^{-n} \| v \| .$$
\end{enumerate}

\end{definition}

It then follows that $E_x^s$ and $E_x^u$ are continuous sub-bundles of $T_xM$. 

\begin{remark}
We can always choose the metric so that $C=1$ in the previous definition: this is called an adapted metric (or a Mather metric) and we will fix one from now on. See \cite{BS02} for a quick proof.
\end{remark}

\begin{definition}
A point $x \in M$ is called non-wandering if, for any open neighborhood $U$ of $x$, $$ U \cap \bigcup_{n > 0} f^n(U)  \neq \emptyset.$$ We denote the set of all non-wandering points by $\Omega(f)$.
\end{definition}

It is easy to check that the non-wandering set of $f$ is a compact invariant subset of $M$. Also, any periodic point of $f$ can be seen to lie in $\Omega(f)$. The definition of Axiom A diffeomorphisms is chosen so that the dynamical system $(f,\Omega(f))$ exhibit a chaotic behavior similar to the one found in symbolic dynamics. Namely:

\begin{definition}
A diffeomorphism $f: M \rightarrow M$ is said to be Axiom A if \begin{itemize}
    \item $\Omega(f)$ is a hyperbolic set for $f$,
    \item $ \Omega(f) = \overline{\{ x \in M \ | \ \exists n > 0, \ f^n(x) = x  \} } . $
\end{itemize} 
\end{definition}

In general, the non-wandering set of an Axiom A diffeomorphism can be written as the union of smaller invariant compact sets. Those are the sets on which we usually work.

\begin{theorem}
One can write $\Omega(f) = \Omega_1 \cup \dots \Omega_k$, where $\Omega_i$ are nonempty compact disjoint sets, such that 
\begin{itemize}
    \item $f(\Omega_i)=\Omega_i$, and $f_{|\Omega_i}$ is topologically transitive,
    \item $\Omega_i = X_{1,i} \cup \dots \cup X_{r_i,i}$ where the $X_{j,i}$ are disjoint compact sets,  $f(X_{j,i}) = X_{j+1,i}$ ($X_{r_i+1,i}=X_{1,i}$) and $f^{n_i}_{| X_{j,i}}$ are all topologically mixing.
\end{itemize}
The sets $\Omega_i$ are called \textbf{basic sets}.
\end{theorem}

\begin{remark}
Notice that any basic set $\Omega$ that satisfies the nonlinearity hypothesis (NL) is necessarily a perfect set. Indeed, if $p \in \Omega$ were an isolated point, then by transitivity $p$ should have dense orbit in $\Omega$, but at the same time be a periodic point. In would imply that $\Omega$ is composed of a unique periodic orbit, which is forbidden by the nonlinearity hypothesis.
\end{remark}

For this section, we fix a basic set (not necessarily an attractor) $\Omega \subset M$ for an Axiom A diffeomorphism $f$. 

\subsection{Equilibrium states}

\begin{definition}

We say that a map $\varphi: X \subset M \rightarrow \mathbb{R}$ is Hölder if there exists $\alpha \in (0,1)$ and $C > 0$ such that $\forall x,y \in X, \ |\varphi(x) - \varphi(y)| \ \leq C d(x,y)^\alpha $, where $d$ is the natural geodesic distance induced by the metric on $M$. We note $C^\alpha(X)$ the set of $\alpha$-Hölder maps on $X$. If $\alpha > 1$, we denote by $C^{\alpha}(M)$ the set of functions that are $\lfloor \alpha \rfloor$ times differentiable, with $(\alpha - \lfloor \alpha \rfloor)$-Hölder derivatives.

\end{definition}

\begin{remark}

A theorem by McShane \cite{Mc34} proves that any $\alpha$-Hölder map defined on a subset of $M$ can always be extended to an $\alpha$-Hölder map on all $M$. Hence, even if for some definitions the potentials $\varphi$ need only to be defined on $\Omega$, we will always be able to consider them as maps in $C^\alpha(M)$ if necessary.

\end{remark}

\begin{definition}{\cite{Bo75}, \cite{Ru78}}
Let $\psi : \Omega \rightarrow \mathbb{R}$ be a Hölder potential.
Define the pressure of $\psi$ by

$$ P(\psi) := \sup_{\mu \in \mathcal{P}_f(\Omega)} \left\{ h_f(\mu) + \int_\Omega \psi d\mu \right\}, $$

where $\mathcal{P}_f(\Omega)$ is the compact set of all probability measures supported on $\Omega$ that are $f$-invariant, and where $h_f(\mu)$ is the entropy of $\mu$. There exists a unique measure $\mu_\psi \in M_f(\Omega)$ such that $$ P(\psi) = h_f(\mu_\psi) + \int_\Omega \psi d\mu_\psi .$$

This measure has support equal to $\Omega$, is ergodic on $(\Omega,f)$, and is called the equilibrium state associated to $\psi$.

\end{definition}

Two particular choices of potentials stands among the others. The first one is the \textbf{constant potential} $\psi := C$. In this case, the equilibrium measure is the \textbf{measure of maximal entropy} (which is known to be linked with the repartition of periodic orbits, see for example \cite{Be16}). \\

Another natural choice is the \textbf{geometric potential} $$ \psi := - \log |\det (df)_{|u} |, $$
where $\det (df)_{|u}$ stands for the determinant of the linear map $ {(df)_x} : E^u_x \rightarrow E^u_{f(x)} $, which is well defined up to a sign since the $E^u$ are equipped with a scalar product. In the case where $\Omega$ is an attractor, that is, if there exists a neighborhood $U$ of $\Omega$ such that $$\bigcap_{n \geq 0} f^n(U) = \Omega,$$ then the associated equilibrium measure is called a \textbf{SRB measure}. In this case, for any continuous function $g:U \rightarrow \mathbb{R}$ and for lebesgue almost all $x \in U$,

$$ \frac{1}{n} \sum_{k=0}^{n-1} g(f^k(x)) \underset{n \rightarrow \infty}{\longrightarrow} \int_{\Omega} g \ d\mu_{SRB} ,$$

which allows us to think of this measure as the \say{physical equilibrium state} of the dynamical system. See \cite{Yo02} for more details on SRB measures, and also the last chapter of $\cite{Bo75}$.

\subsection{Stable/unstable laminations, bracket and holonomies}

In this subsection, we will recall some results about the existence of stable/unstable laminations, some regularity results in our particular case, and the consequence on the regularity of the holonomies.

\begin{definition}

Let $x \in \Omega$. For $\varepsilon>0$ small enough, we define the local stable and unstable manifold at $x$ by $$ W^s_{\epsilon}(x) := \{ y \in M \ | \ \forall n \geq 0, \ d(f^n(x),f^n(y)) \leq \varepsilon \}, $$
$$ W^u_\varepsilon(x) := \{ y \in M \ | \ \forall n \leq 0, \ d(f^n(x),f^n(y)) \leq \varepsilon \} .$$
We also define the global stable and unstable manifolds at $x$ by
$$ W^s(x) := \{ y \in M \ |  \ d(f^n(x),f^n(y)) \underset{n \rightarrow +\infty}{\longrightarrow} 0 \} ,$$
$$ W^u(x) := \{ y \in M \ | \ d(f^n(x),f^n(y)) \underset{n \rightarrow -\infty}{\longrightarrow} 0 \}.  $$

\end{definition}

\begin{theorem}[\cite{BS02}, \cite{KH95}, \cite{Bo75}]

Let $f$ be a $\mathcal{C}^r$ Axiom A diffeomorphism and let $\Omega$ be a basic set.
For $\varepsilon>0$ small enough and for $x \in \Omega$:
\begin{itemize}
    \item $W^s_\varepsilon(x)$ and $W^u_\varepsilon(x)$  are $\mathcal{C}^r$ embedded disks,
    \item $\forall y \in \Omega \cap W^s_\varepsilon(x), \ T_y W^s_\varepsilon (x) = E^s_y$,
    \item $\forall y \in \Omega \cap W^u_\varepsilon(x), \ T_y W^u_\varepsilon (x) = E^u_y$,
    \item $f(W^s_\varepsilon(x)) \subset \subset W^s_\varepsilon(f(x))$ and $ f(W^u_\varepsilon(x)) \supset \supset W^u_\varepsilon(f(x)) $ where $\subset \subset$ means \say{compactly included},
    \item $ \forall y \in W^s_\varepsilon(x), \ \forall n \geq 0,  \ d^s(f^n(x),f^n(y)) \leq  \kappa^n d^s(x,y) $ where $d^s$ denotes the geodesic distance on the submanifold $W^s_\varepsilon$,
    \item $ \forall y \in W^u_\varepsilon(x), \ \forall n \geq 0, \ d^u(f^{-n}(x),f^{-n}(y)) \leq  \kappa^{n} d^u(x,y) $ where $d^u$ denotes the geodesic distance on the submanifold $W^u_\varepsilon$.

\end{itemize}

Moreover
$$ \bigcup_{n \geq 0} f^{-n}(W^s_\varepsilon(f^n(x))) = W^s(x)  $$
and $$ \bigcup_{n \geq 0} f^{n}(W^u_\varepsilon(f^{-n}(x))) = W^u(x),$$
and so the global stable and unstable manifolds are injectively immersed manifolds in $M$.

\end{theorem}

The family $(W^s_\varepsilon(x),W^u_\varepsilon(x))_{x \in \Omega}$ forms two transverse continuous laminations, which allows us to define the so-called bracket and holonomies maps. 

\begin{definition}[\cite{Bo75}, \cite{KH95}, \cite{BS02}]
For $\varepsilon>0$ small enough, there exists $\delta>0$ such that $W^s_\varepsilon(x) \cap W^u_\varepsilon(y)$ consists of a single point $[x,y]$ whenever $x,y \in \Omega$ and $d(x,y) < \delta$. In this case, $[x,y] \in \Omega$, and the map
$$ [\cdot, \cdot] : \{ (x,y) \in \Omega \times \Omega \ , \ d(x,y) < \delta \} \longrightarrow \Omega $$
is continuous. Moreover, there exists $C>0$ such that
$$ d^s([x,y],x) \leq C d(x,y), \quad \text{and} \quad d^u([x,y],y) \leq C d(x,y).$$

\end{definition}

\begin{definition} Fix $x,y$  two close enough points in $\Omega$ lying in the same local stable manifold. Let $U^u \subset W^u_\varepsilon(x) \cap \Omega$ be a small open neighborhood of $x$ relatively to $ W^u_\varepsilon(x) \cap \Omega$. The map
$$ \pi_{x,y} : U^u \subset W^u_{\varepsilon}(x) \cap \Omega \longrightarrow W^u_{\varepsilon}(y) \cap \Omega .$$
defined by $\pi_{x,y}(z) := [z,y]$ is called a stable holonomy map. One can define an unstable holonomy map similarly on pieces of local stable manifolds intersected with $\Omega$. Since the stable and unstable laminations are Hölder regular in the general case, those honolomy maps are only Hölder regular in general (\cite{KH95}, theorem 19.1.6).
\end{definition}

To work, we will need a bit of regularity on the unstable holonomies. Fortunately for us, in the particular case where the stable lamination are codimension 1, we have some regularity results: this is Theorem 1 page 25 of \cite{Ha89}, and Theorem  19.1.11 in \cite{KH95} for Anosov diffeomorphisms. For general hyperbolic sets, the proof is done is \cite{PR02}. 

\begin{theorem}

Let $f$ be an Axiom A diffeomorphism, and let $\Omega$ be a basic set. Suppose that $f$ has codimension one stable laminations, that is, $\dim E^u_x = 1$ for all $x \in \Omega$. Then the stable lamination is $C^{1+\alpha}$ for some $\alpha>0$. In particular, the stable holonomies maps are $C^{1 + \alpha}$ diffeomorphisms.
\end{theorem}

\begin{remark}
Here, the holonomies being $C^{1+\alpha}$ means that the map $\pi_{x,y}$ extends to small curves $W^u_\delta(x) \longrightarrow W^u_\varepsilon(y)$, and that the extended map is a $C^{1+\alpha}$ diffeomorphism.
\end{remark}

From now on, we will work under the assumption that $f$ has codimension one stable laminations. This is always true if $\dim M =2$, or if $\dim M=3$ by exchanging $f$ by $f^{-1}$ if necessary.  We fix a Hölder potential $\psi:\Omega \rightarrow \mathbb{R}$ and its associated equilibrium state $\mu$.

\subsection{Markov partitions}

In this subsection, we will construct the topological space on which we will work in this paper: the hypothesis made on the dimension of the unstable lamination will allow us to approximate the dynamics by a dynamical system on a finite disjoint union of smooth curves. For this we need to recall some results about Markov partitions.

\begin{definition}

A set $R \subset \Omega$ is called a \emph{rectangle} if 
$$ \forall x,y \in R, \ [x,y] \in R .$$

A rectangle is called proper if $\overline{\text{int}_\Omega(R)} = R$. If $x \in R$, and if $\text{diam}(R)$ is small enough with respect to $\varepsilon$, we define $$ W^s(x,R) := W^s_\varepsilon(x) \cap R  \quad \text{and} \quad W^u(x,R) := W^u_\varepsilon(x) \cap R .$$

\end{definition}

\begin{remark}
Notice that a rectangle isn't always connected, and might even have an infinite number of connected component, even if $\Omega$ is itself connected. This technicality is noticed in \cite{Pe19}, at the first paragraph of subsection 3.3, and is an obstruction to the existence of finite Markov partition with connected elements.
\end{remark}

\begin{definition}
A Markov partition of $\Omega$ is a finite covering $\{R_a\}_{a \in \mathcal{A}}$ of $\Omega$ by proper rectangles such that
\begin{itemize}
    \item $\text{int}_\Omega R_a \cap \text{int}_\Omega R_b = \emptyset$ if $a \neq b$
    \item $f\left( W^u(x,R_a) \right) \supset W^u(f(x),R_b)$ and $f\left( W^s(x,R_a) \right) \subset W^s(f(x),R_b)$ \\ when $x \in \text{int}_\Omega R_a \cap f^{-1}\left(\text{int}_\Omega R_a \right)$. 
\end{itemize}
\end{definition}

\begin{theorem}[\cite{Bo75},\cite{KH95}]
Let $\Omega$ be a basic set for an Axiom A diffeomorphism $f$. Then $\Omega$ has Markov partitions of arbitrary small diameter.

\end{theorem}

From now on, we fix once and for all a Markov partition $\{R_a\}_{a \in \mathcal{A}}$ of $\Omega$ with small enough diameter. Remember that, since $\Omega$ is not an isolated cycle, it is a perfect set. In particular, $\text{diam} R_a >0$ for all $a \in \mathcal{A}$. 

\begin{definition}
We fix for the rest of this paper some periodic points $x_a \in \text{int}_\Omega R_a$ for all $a \in \mathcal{A}$ (this is possible by density of such points in $\Omega$). By periodicity, $x_a \notin W^u(x_b)$ when $a \neq b$. 
\end{definition}

\begin{definition}

We set, for all $a \in \mathcal{A}$, $$ S_a := W^s(x_a,R_a) \quad \text{and} \quad U_a := W^u(x_a,R_a) .$$
They are closed sets included in $R_a$, and are defined so that $[U_a,S_a] = R_a$. They will allow us to decompose the dynamics into a stable and an unstable part. Notice that the decomposition is unique:
$$ \forall x \in R_a, \exists ! (y,z) \in U_a \times S_a, \ x=[y,z] .$$

\end{definition}

The intuition of the construction to come is that, after a large enough number of iterates, $f^n$ can be approximated by a map that is only defined on the $(U_a)_{a \in \mathcal{A}}$.

\subsection{A factor dynamics}

In this section we construct what will take the role of the shift map in our context. The construction is inspired by the symbolic case and already appear in the work of Dolgopyat \cite{Do98}. 

\begin{notations}
Let $a$ and $b$ be two letters in $\mathcal{A}$.
We note $a \rightarrow b$ if $f(\text{int}_\Omega R_a) \cap \text{int}_\Omega R_b \neq \emptyset $.

\end{notations}

\begin{definition}

We define $$\mathcal{R} := \bigsqcup_{a \in A} R_a , \quad \mathcal{S} := \bigsqcup_{a \in A} S_a , \quad \mathcal{U} := \bigsqcup_{a \in A} U_a $$
where $\bigsqcup$ denote a formal disjoint union. we also define $$ \mathcal{R}^{(0)} := \bigsqcup_{a \in \mathcal{A}} \text{int}_\Omega R_a \subset \mathcal{R} $$
and $$ \mathcal{R}^{(1)} := \bigsqcup_{a \rightarrow b} (\text{int}_{\Omega} R_a) \cap f^{-1}(\text{int}_{\Omega} R_b) \subset \mathcal{R}^{(0)} $$
So that the map $f:\Omega \rightarrow \Omega$ may be naturally seen as a map $f : \mathcal{R}^{(1)} \longrightarrow \mathcal{R}^{(0)}$. We then define 
$$ \mathcal{R}^{(k)} := f^{-k}( \mathcal{R}^{(0)}) $$
and, finally, we denote the associated residual set by
$$ \widehat{\mathcal{R}} := \bigcap_{k \geq 0} \mathcal{R}^{(k)} $$
so that $f : \widehat{\mathcal{R}} \longrightarrow \widehat{\mathcal{R}}$. 
Seen as a subset of $\Omega$, $\widehat{\mathcal{R}}$ as full measure, by ergodicity of the equilibrium measure $\mu$. Hence $\mu$ can naturally be though as a probability measure on $\widehat{\mathcal{R}}$. 
\end{definition}

\begin{definition}

Let $\mathcal{R}/\mathcal{S}$ be the topological space defined by the equivalence relation $ x \sim y  \Leftrightarrow \exists a \in \mathcal{A}, \ y \in W^s(x,R_a) $ in $\mathcal{R}$. Let $\pi:\mathcal{R} \rightarrow \mathcal{R}/\mathcal{S}$ denotes the natural projection. The map $f:\widehat{\mathcal{R}} \rightarrow \widehat{\mathcal{R}}$ induces a factor map $F : \widehat{\mathcal{R}}/\mathcal{S} \rightarrow \widehat{\mathcal{R}}/\mathcal{S}$. Moreover, the measure $\nu := \pi_* \mu$ is a $F$-invariant probability measure on $\mathcal{R}/\mathcal{S}$.

\end{definition}

\begin{proof}

We just have to check that $f : \mathcal{R}^{(1)} \longrightarrow \mathcal{R}^{(0)}$ satisfy $$ f W^s(x,R_a) \subset W^s(f(x),R_b)   $$ for $x \in (\text{int}_{\Omega} R_a) \cap f^{-1}(\text{int}_{\Omega} R_b)$. This is true by definition of Markov partitions. The induced map satisfy $F \circ \pi = \pi \circ f$, and so $F_* \nu = \nu$. \end{proof}

\begin{remark}

There is a natural isomorphism $  \mathcal{U} \simeq \mathcal{R}/\mathcal{S}$ that is induced by the inclusion $\mathcal{U} \hookrightarrow \mathcal{R}$. This allows us to identify all the precedent construction to a dynamical system on $\mathcal{U}$. Namely:

\begin{itemize}
    \item The projection $\pi: \mathcal{R} \rightarrow \mathcal{R}/\mathcal{S}$ is identified with
    $$ \begin{array}[t]{lrcl}
 \pi: & \mathcal{R} \quad & \longrightarrow & \quad \mathcal{U} \\
    & x \in R_a & \longmapsto &  [x,x_a] \in U_a  \end{array}  $$
    
    \item The factor map $F:\widehat{\mathcal{R}}/\mathcal{S} \rightarrow \widehat{\mathcal{R}}/\mathcal{S}$ is identified with
    $$ \begin{array}[t]{lrcl}
 F: & \widehat{\mathcal{U}} \quad \quad \quad & \longrightarrow & \quad \widehat{\mathcal{U}} \\
    & x \in R_a \cap f^{-1}(R_b)  & \longmapsto &  [f(x),x_b] \in U_b  \end{array}   $$
    where $\widehat{\mathcal{U}}$ is defined similarly to $\widehat{\mathcal{R}}$, but with $F$ replacing $f$ in the construction.
    
    \item The measure $\nu$ is identified to the unique measure on $\mathcal{U}$ such that:
    
    $$ \forall h \in C^0(\mathcal{R},\mathbb{R}) \ \text{S-constant}, \ \int_\mathcal{U} h d\nu \ = \ \int_\mathcal{R} h d\mu $$
    
    where $S$-constant means that $\forall a \in \mathcal{A}, \forall x,y \in R_a, \ x \in W^s(y,R_a) \Rightarrow h(x)=h(y)$.
    
\end{itemize}

\end{remark}

\begin{remark}
Since our centers $x_a$ are periodic, $x_a \in \widehat{\mathcal{R}}$, and hence $x_a \in \widehat{\mathcal{U}}$.
\end{remark}

\begin{lemma}

Under the nonlinearity hypothesis, the set $\widehat{\mathcal{U}}$ is a perfect set. In particular, $\text{diam} \ U_a > 0$ for all $a \in \mathcal{A}$.

\end{lemma}

\begin{proof}

First of all, we prove that $\widehat{\mathcal{U}}$ is infinite. By hypothesis, there exists an infinite family of distinct periodic orbits for $f$ in $\Omega$. Since $f$ is Axiom A, the set of periodic orbits for $f$ is dense in $\Omega$. In particular, the set of $f$-periodic points in $\mathcal{R}^{(0)}$ is infinite. \\

Let $x \neq y \in R_a$ be two $f$-periodic points with distinct periodic orbits. Then $\pi(x)$ and $\pi(y)$ are $F$-periodic with distinct periodic orbits. (In particular, they are in $\widehat{\mathcal{U}}$.) Indeed, if there existed some $n_0$ and $m_0$ such that $\pi(f^{n_0}(x)) = \pi(f^{m_0}y) $, then by definition of $\pi$ we would have $ f^{n_0}(x) \in W^s(f^{m_0} y) $. Hence $d(f^{n+n_0}(x), f^{n+m_0} (y)) \underset{n \rightarrow + \infty}{\longrightarrow} 0$, a contradiction.

So $\widehat{\mathcal{U}}$ is an infinite set. Now we justify that $F_{|\widehat{\mathcal{U}}}$ is topologically transitive. 

We see by the Baire category theorem that $\widehat{\mathcal{R}}$ is a dense subset of $\Omega$. Moreover, $ f(\widehat{\mathcal{R}}) = \widehat{\mathcal{R}} $. It then follows from the transitivity of $f_{|\Omega}$ that $f_{|\widehat{\mathcal{R}}}$ is topologically transitive. Hence, $ F_{| \widehat{\mathcal{U}}} $ is topologically transitive, as a factor of $f_{|\widehat{\mathcal{R}}}$. We conclude the proof using the argument of remark 2.2. \end{proof}

\begin{remark}
The fact that $\widehat{\mathcal{U}}$ is perfect, combined with the fact that holonomies extend to $C^{1+\alpha}$ maps, allows us to consider for $x \in \mathcal{R}$ the (absolute value of the) derivatives along the unstable direction of the holonomies in a meaningful way (without having to chose an extension). We denote them by $ \partial_u \pi(x) $. (Formally, $|\partial_u \pi(x)| := |(d\pi)_x(\vec{n})|$ where $\vec{n}$ is a unit vector in $E^u_x$ and where $|\cdot|$ is the norm on $E^u_x$.) \\

It is then known that those quantities are uniformly bounded $\cite{PR02}$. In other words:

$$ \exists c,C > 0, \forall x \in \mathcal{R}, \ c < |\partial_u \pi(x)| \leq C .$$
\end{remark}

\begin{remark}
The map $F : \widehat{\mathcal{U}} \rightarrow \widehat{\mathcal{U}}$ is an eventually expanding map. \\

Indeed, let $n \geq 0$. By Theorem 2.2, if $x,y \in \widehat{\mathcal{U}}$ are close enough depending on $n$, then $d(f^n(x),f^n(y)) \geq \kappa^{-n} d(x,y)$ and $f^n(x),f^n(y)$ are still in the same $R_b$. Since $F^n = \pi \circ f^n$, and since $\pi$ has bounded derivatives, we get:

$$\exists c>0, \ \forall n, \ \forall (x,y) \ \text{close enough} , \  d^u(F^n x, F^n y) \geq c \kappa^{-n} d^u(x,y) .$$

\end{remark}

\subsection{A transfer operator}

The dynamical system $(F,\widehat{\mathcal{U}},\nu)$ is constructed to behave like a semi-shift, and so it is natural to search for a transfer operator defined on $\mathcal{U}$ that leaves $\nu$ invariant. Recall that $\mu$ is the equilibrium measure for some Hölder potential $\psi$ on $M$: the fact that $\psi$ isn't necessarily $S$-constant suggest that we will have to search for another potential, cohomologous to $\psi$, that is $S$-invariant. Looking at the symbolic case, we notice from the proof of the proposition 1.2 in \cite{PP90} that the potential $\varphi_0 \in C^\alpha( \mathcal{R},\mathbb{R} )$ (take $\alpha$ smaller if necessary) defined by the formula $$  \varphi_0(x) := \psi(\pi(x)) + \sum_{n=0}^\infty \left( \psi( f^{n+1} \pi x ) - \psi(f^n \pi f x) \right) \quad \quad \quad  $$ might be interesting to consider. First of all, recall what is a transfer operator.

\begin{definition}

For some Hölder potential $\varphi:\mathcal{U} \rightarrow \mathbb{R}$, define the associated \emph{transfer operator} $ \mathcal{L}_\varphi : C^0( \mathcal{U}^{(0)} ,\mathbb{C}) \longrightarrow  C^0( \mathcal{U}^{(1)} ,\mathbb{C}) $
by 
$$ \forall x \in \mathcal{U}^{(0)}, \  \mathcal{L}_\varphi h (x) := \sum_{y \in F^{-1}(x)} e^{\varphi(y)} h(y) .$$

Iterating $\mathcal{L}_\varphi$ gives
$$ \forall x \in \mathcal{U}^{(0)}, \ \mathcal{L}_\varphi^n h (x) = \sum_{y \in F^{-n}(x)} e^{S_n \varphi(y)} h(y) ,$$
where $S_n \varphi(z) := \sum_{k=0}^{n-1} \varphi(F^k(z))$ is a Birkhoff sum. \\

By duality, $\mathcal{L}_\varphi$ also acts on the set of measures on $\mathcal{U}$. If $m$ is a measure on $\mathcal{U}^{(1)}$, then $\mathcal{L}_\varphi^* m$ is the unique measure on $\mathcal{U}^{(0)}$ such that

$$ \forall h \in C^0( \mathcal{U}^{(1)}, \mathbb{C}), \ \int_{\mathcal{U}} h \ d \mathcal{L}_\varphi^* m = \int_{\mathcal{U}} \mathcal{L}_\varphi h \ dm  .$$

\end{definition}

\begin{remark}
We may rewrite the definition by highlighting the role of the inverse branches of $F$.
For some $a \rightarrow b$, we see that $F : U_{ab}^{(1)} \rightarrow U_{b}^{(0)}$ is (the restriction of) a diffeomorphism. We denote by $g_{ab} : U_b \rightarrow U_{ab}$ its local inverse, that can be defined by the formula 
$$ \forall x \in U_b, \ g_{ab}(x) := f^{-1}([ x , f(x_a) ]) .$$
Then the transfer operator can be rewritten as follow:

$$ \forall x \in U_b, \ \mathcal{L}_\varphi h(x) = \sum_{a \rightarrow b} e^{\varphi(g_{ab}(x))} h(g_{ab}(x)) .$$

\end{remark}

\begin{theorem}

There exists a Hölder function $h: \mathcal{U} \rightarrow \mathbb{R}$ such that $\mathcal{L}_{\varphi_0}^* (h \nu) = e^{P(\psi)} h \nu $, where $$\varphi_0(x) := \psi(\pi(x)) + \sum_{n=0}^\infty \left( \psi( f^{n+1} \pi x ) - \psi(f^n \pi f x) \right).$$

\end{theorem}

\begin{proof}

This fact follows from the symbolic setting, but a quick and clean proof can also be achieved using the following geometrical results, extracted from \cite{Le00} and \cite{Cl20}. See in particular Theorem 3.10 in \cite{Cl20}. Define, when $x \in W^s(y)$, $$\omega^+(x,y) := \sum_{n=0}^\infty \left( \psi(f^n(x)) - \psi(f^n(y)) \right),$$
and when $x \in W^u(y)$, $$ \omega^-(x,y) := \sum_{n=0}^\infty \left( \psi(f^{-n}(x)) - \psi(f^{-n}(y)) \right). $$
There exists two families of (nonzero and) finite measures $m_x^s$, $m_x^u$ indexed on $x \in \mathcal{R}^{(0)}$ such that $\text{supp} ( m_x^u) \subset W^u(x,R_a)$ and $\text{supp} ( m_x^s ) \subset W^s(x,R_a)$ if $x \in R_a$, and that satisfies the following properties:

\begin{itemize}
    \item If $\pi$ is a stable holonomy map between $W^u(y,R_a)$ and $W^u(\pi(y),R_a)$, then
    $$ \frac{d\left(\pi_*m_y^u\right)}{d m_{\pi(y)}^u}(\pi(x)) = e^{\omega^+(x,\pi x)}.$$
    and a similar formula holds for the $m_y^s$ with an unstable holonomy map.
    \item The family of measures is $f$-conformal:
    $$ \frac{ d \left( f_* m_y^u \right)}{d m_{f(y)}^u}(f(x)) = e^{\psi(x) - P(\psi) } $$
    and a similar formula holds for the $m^s_y$.
\end{itemize}

Finally, this family of measures is related to $\mu$ in the following way. There exists positive constants $(c_a)_{a \in \mathcal{A}}$ such that, for any measurable map $g:M \rightarrow \mathbb{C}$, the following formula holds:

$$ \int_M g d\mu = \sum_{a \in \mathcal{A}} \int_{U_a} \int_{S_a} e^{\omega^+([x,y],x)+\omega^-([x,y],y)} g([x,y]) c_a dm_{x_a}^s(y) dm_{x_a}^u(x)  .$$

From this we can link $\nu$ and the family $m_{x_a}^u$ in the following way. Define $h_0 : \mathcal{U} \rightarrow \mathbb{R}$ by the formula
$$ \forall a \in \mathcal{A}, \ \forall x \in U_a, \ h_0(x) := c_a \int_{S_a} e^{\omega^-([x,y],y) + \omega^+([x,y],y)} dm_{x_a}^s(y) > 0.$$

Then, for any $S$-constant map $g:\mathcal{R} \rightarrow \mathbb{R}$, we have
$$ \int_\mathcal{U} g d\nu = \int_M g d\mu = \sum_{a \in \mathcal{A}} \int_{U_a} g(x) h_0(x) dm_{x_a}^u(x) .$$

It follows from the properties of the bracket and the fact that $\psi$ is supposed Hölder that $h_0$ is Hölder. So $h:=h_0^{-1}$ is Hölder and satisfy $h \nu = m^u$, where $m^u_{|U_a} := m^u_{x_a}$. We check that $m^u$ is an eigenmeasure for $\mathcal{L}_{\varphi_0}$. 
Fix $a \rightarrow b \in \mathcal{A}$, and $G: U_b \rightarrow \mathbb{C} $ continuous.
We have
$$ \int_{U_b} G(g_{ab}(x)) d m_{x_b}^u(x) = \int_{U_b} G \left( f^{-1} [x,f(x_a)] \right)  d m_{x_b}^u(x) $$
$$ = \int_{f(U_{ab})} G \left( f^{-1} x \right) e^{\omega^+(x,\pi x)} d m_{f(x_a)}^u(x)  $$
$$ = \int_{U_{ab}} G(x) e^{\omega^+(f x, \pi f x)} e^{-\psi(x)+P(\psi)} d m^u_{x_a}(x) .$$
$$ = e^{P(\psi)} \int_{U_{ab}} G(x) e^{-\varphi_0(x)} dm_{x_a}^u .$$

In particular, for any continuous $G: \mathcal{U} \rightarrow \mathbb{R}$:
$$ \int_{\mathcal{U}} \mathcal{L}_{\varphi_0}(G) dm^u = \sum_{b \in \mathcal{A}} \int_{U_b} \sum_{a \rightarrow b} e^{\varphi_0( g_{ab}(x) ) } G(g_{ab}(x)) d m^u_{x_a}  $$
$$ = e^{P(\psi)} \underset{a \rightarrow b}{\sum_{a, b \in \mathcal{A}}} \int_{U_{ab}} G(x) dm^u_{x_a} = \int_{\mathcal{U}} G dm^u. $$
Which concludes the proof. \end{proof}

\begin{lemma}
Define $\varphi := \varphi_0 - \ln h \circ F + \ln h - P(\psi)$. Then $\varphi:\mathcal{U} \rightarrow \mathbb{R}$ is Hölder and normalized, that is: $ \mathcal{L}_\varphi^* \nu = \nu $, $ \mathcal{L}_\varphi 1 = 1 $ on $\widehat{\mathcal{U}}$. Moreover, $\varphi \leq 0$ on $\widehat{\mathcal{U}}$, and there exists an integer $N \geq 0$ such that $S_N \varphi<0$ on $\widehat{\mathcal{U}}$.
\end{lemma}

\begin{proof}

It follows from the definition of $\varphi$ that $  \mathcal{L}_\varphi^* \nu = \nu $. Then, we see that $ \mathcal{L}_\varphi 1 = 1 $ on $L^2(\nu) $. Indeed, for any $L^2(\nu)$ map $G$, we have $$\int_{\mathcal{U}} \mathcal{L}_\varphi(1) G d\nu = \int_{\mathcal{U}}\mathcal{L}_\varphi( G \circ F  ) d\nu = \int_{\mathcal{U}} G \circ F d\nu = \int_{\mathcal{U}} G d\nu .$$

Since $\mathcal{L}_\varphi 1$ and $1$ are continuous functions on $\widehat{\mathcal{U}}$, and since $\nu$ has full support, it follows that $\mathcal{L}_\varphi 1 = 1$ on all $\widehat{\mathcal{U}}$, and even on $\mathcal{U}$ by density. We check that $\varphi$ is eventually negative on the residual set. First of all, the equality $ 1 = \mathcal{L}_\varphi(1)$ implies that $\varphi \leq 0$ on $\widehat{\mathcal{U}}$. Then, since $F$ is uniformly expanding, there exists some integer $N$ such that $F^{-N}(x)$ always contains at least two distinct points, for any $x$. The equality
$ 1 = \mathcal{L}_\varphi^{N}(1)$ then allows us to conclude that $S_N \varphi<0$ on $F^{-N}(\widehat{\mathcal{U}})=\widehat{\mathcal{U}}$. \end{proof}

\subsection{Some regularity for $\nu$}

In this subsection, we introduce the usual symbolic formalism, recall the Gibbs estimates for $\nu$, and prove an upper regularity result for $\nu$. In particular, we prove that $\nu$ don't have atoms under our non-linearity condition. Notice that if $\Omega$ were an isolated cycle, then the measure of maximal entropy would be the Dirac on the cycle, which is fully discrete, so this is not completely trivial.

\begin{notations}

For $n \geq 1$, a word $\mathbf{a} = a_1 \dots a_n \in \mathcal{A}^n$ is said to be admissible if $a_1 \rightarrow a_2 \rightarrow \dots \rightarrow a_n$. We define:
\begin{itemize}
    \item $\mathcal{W}_n := \{ \mathbf{a} \in \mathcal{A}^n \ | \ \mathbf{a} \ \text{is admissible} \ \}$.
    \item For $\mathbf{a} \in \mathcal{W}_n$, define $g_{\mathbf{a}} := g_{a_1 a_2}  g_{a_2 a_3}  \dots g_{a_{n-1} a_{n}} : U_{a_n} \longrightarrow U_{a_{1}}$. 
    \item For $\mathbf{a} \in \mathcal{W}_n$, define $U_{\mathbf{a}} := g_{\mathbf{a}}\left( U_{a_n} \right)$, $ U_{\mathbf{a}} := g_{\mathbf{a}}\left( U_{a_n} \right) $, and $ \widehat{U}_{\mathbf{a}} := g_{\mathbf{a}}\left( \widehat{\mathcal{U}} \cap U_{a_n} \right) $ 
    \item For $\mathbf{a} \in \mathcal{W}_n$, set $x_{\mathbf{a}} := g_{\mathbf{a}}\left( x_{a_n} \right) \in \widehat{U}_{\mathbf{a}} $. 
\end{itemize}

\end{notations}

\begin{remark}

Since $F$ is eventually expanding, the maps $g_{\mathbf{a}}$ are eventually contracting as $n$ becomes large. As $\mathcal{U}$ is included in a finite union of unstable curve, that are one dimensional riemannian manifolds, it makes sense to consider absolute value of the derivatives of $F$ and $g_{\mathbf{a}}$, and we will do it from now on. For points in $\mathcal{U}$, this correspond to the absolute value of the derivative in the unstable direction. \\

Since the holonomies are uniformly bounded, and by hyperbolicity of $f$, we see that $$ \forall x \in \widehat{\mathcal{U}}, |\partial_u F^n(x)| \geq c \kappa^{-n} .$$

Consequently,
$$ \forall \mathbf{a} \in \mathcal{W}_n, \ \forall x \in U_{a_n}, \ |g_{\mathbf{a}}'(x)| \leq c^{-1} \kappa^n  . $$

Moreover, using remark 2.10, we see that

$$\forall \mathbf{a} \in \mathcal{W}_n,  \text{diam}(U_{\textbf{a}}) = \text{diam} \ g_{\mathbf{a}}\left( U_{a_n} \right) \leq c^{-1} \kappa^n .$$

A consequence for our potential is that it has $\textbf{exponentially vanishing variations}$. Namely, since $\varphi$ is Hölder, the following holds:

$$ \exists C>0, \ \forall n \geq 1, \ \forall \mathbf{a} \in \mathcal{W}_n, \ \forall x,y \in U_{\mathbf{a}}, \ |\varphi(x) - \varphi(y)| \leq C \kappa^{\alpha n} .$$

\end{remark}

\begin{remark}

For a fixed $n$, the family $(U_{\mathbf{a}})_{\mathbf{a} \in W_n}$ is a partition of $\mathcal{U}$ (modulo a boundary set of zero measure). In particular, for any continuous map $g:\mathcal{U} \rightarrow \mathbb{C}$, we can write
$$ \int_{\mathcal{U}} g d \nu = \sum_{\mathbf{a} \in \mathcal{W}_n} \int_{U_{\mathbf{a}}} g d\nu = \sum_{\mathbf{a} \in \mathcal{W}_n} \int_{U_{\mathbf{a}}} g d\nu .$$

\end{remark}

\begin{lemma}[Gibbs estimates, \cite{PP90}]

$$ \exists C_0>1, \ \forall n \geq 1, \ \forall \mathbf{a} \in \mathcal{W}_n, \  \forall x \in U_{\mathbf{a}}, \ \ C_0^{-1} e^{S_n \varphi(x)} \leq \nu( U_{\mathbf{a}} ) \leq C_0 e^{S_n \varphi(x)} $$

\end{lemma}

 \begin{proof}

We have
$$ \int_{\mathcal{U}} e^{- \varphi} \mathbb{1}_{U_{a_1 \dots a_n}} d\nu_\varphi =  \int_{\mathcal{U}} \mathcal{L}_\varphi\left( e^{- \varphi} \mathbb{1}_{U_{a_1 \dots a_n}} \right) d\nu  = \int_{\mathcal{U}} \mathbb{1}_{U_{a_2 \dots a_n}} d\nu =  \nu(U_{a_2 \dots a_n}). $$
Moreover, since $\varphi$ has exponentially decreasing variations, we can write that:
$$\forall x \in U_{a_1 \dots a_n}, \ e^{-\varphi(x) - C \kappa^{\alpha n}} \nu(U_{a_1 \dots a_n}) \leq \int_{\mathcal{U}} e^{- \varphi} \mathbb{1}_{U_{a_1 \dots a_n}} d\nu \leq  e^{-\varphi(x) + C \kappa^{\alpha n}} \nu(U_{a_1 \dots a_n})  $$
And so $$ \forall x \in U_{a_1 \dots a_n}, \ e^{-\varphi(x) - C \kappa^{\alpha n}}  \leq \frac{\nu(U_{a_2 \dots a_n})}{\nu(U_{a_1 \dots a_n}) } \leq  e^{-\varphi(x) + C \kappa^{\alpha n}} . $$
Multiplying those inequalities gives us the desired relation.
 \end{proof}

\begin{lemma}

The measure $\nu$ is upper regular, that is, there exists $C>0$ and $\delta_{\text{up}}>0$ such that

$$ \forall x \in \mathcal{U}, \ \forall r > 0, \ \nu( B(x,r) ) \leq C r^{\delta_{\text{up}}}.$$

Where $B(x,r)$ is a ball of center $x$ and radius $r$ in $\mathcal{U}$.

\end{lemma}

\begin{proof}

Let $x \in \mathcal{U}$, and let $r>0$ be small enough. If $B(x,r) \cap \widehat{\mathcal{U}}= \emptyset$ then $\nu(B(x,r)) = 0$ since $\widehat{\mathcal{U}}$ has full measure, and the proof is done. If $B(x,r) \cap \widehat{\mathcal{U}} \neq \emptyset$, then let $\tilde{x} \in B(x,r) \cap \widehat{\mathcal{U}}  $. We then see that $ \nu(B(x,r)) \leq \nu(B(\tilde{x},2r)$. \\

Recall that $F$ has bounded derivative on $\mathcal{U}$: $|F'|_\infty < \infty $. In particular, there exists a constant $\kappa_1 \in (0,1)$ such that $ |(F^n)'|_\infty \leq \kappa_1^{-n} $. Hence, $ \forall n \geq 1, \ \forall \mathbf{a} \in \mathcal{W}_n, \ \text{inf}|g_{\mathbf{a}}'| > \kappa_1^n $. In particular, $ \text{diam}(U_{\mathbf{a}}) > \kappa_1^n \eta_0$, where $\eta_0 := \min_{a \in \mathcal{A}} \text{diam} (U_a) > 0$ (This is where we use the fact that $\widehat{\mathcal{U}}$ is perfect). \\

Recall also from the lemma 2.7 that there exists an integer $N$ such that $S_N \varphi < 0$ on $\widehat{\mathcal{U}}$.  We then let $n(r)$ be the unique integer $n$ such that

$$ \eta_0 \kappa_1^{N(n+1)} < 2r \leq  \eta_0 \kappa_1^{Nn} .$$
Then, by construction, $ 2 r \leq \text{diam}(U_{\mathbf{a}}) , \forall \mathbf{a} \in \mathcal{W}_{n(r)}$. Since everything is one dimensional, it follows that $B(\tilde{x},2r)$ is contained in at most two sets of the form $U_{\mathbf{a}}$, $\mathbf{a} \in \mathcal{W}_n$. In particular, by the Gibbs estimates, there exists eventually a $y \in \widehat{\mathcal{U}}$ such that

$$ \nu\left( B(x,r) \right) \leq C_0 e^{S_{N n(r)} \varphi(x)} +  C_0 e^{S_{N n(r)} \varphi(y)} .$$

But then notice that $ S_{n N} \varphi(x) < - n \inf_{\widehat{\mathcal{U}}}|S_N \varphi| $, and similarly for $S_{n N} \varphi(y) $. Hence:

$$  \nu\left( B(x,r) \right) \leq 2 C_0 e^{- n(r) \inf_{\widehat{\mathcal{U}}}|S_N \varphi|} \leq C r^{\delta_{\text{up}}}$$

for some $C>0$ and $\delta_{\text{up}} >0$. \end{proof}

In particular, it follows that $\nu$ don't have atoms, which is some good news since we are willing to prove Fourier decay. Moreover, we get a regularity result on $\mu$ (that holds more generally if one replace $(NL)$ by \say{$\widehat{\mathcal{U}}$ is perfect}).

\begin{corollary}
Let $f$ be an Axiom A diffeomorphism, and let $\Omega$ be a basic set with codimension one stable foliation. Suppose that $\Omega$ satisfies (NL). Then any equilibrium state $\mu$ associated to Hölder potential $\psi : \Omega \rightarrow \mathbb{R}$ is upper regular.
\end{corollary}

\begin{proof}
Let $x \in \Omega$ and $r > 0$ be small enough. Write $B(x,r) = \bigcup_{a \in \mathcal{A}} B(x,r) \cap R_a$. Notice that $B(x,r) \cap R_a \subset \{ y \in R_a \ | \ \exists z \in  B(x,r) \cap R_a, \ \pi(y) = \pi(z)\}$, and so:

$$ \mu(B(x,r)) = \sum_{a \in \mathcal{A}} \mu \left( B(x,r) \cap R_a \right) \leq \sum_{a \in \mathcal{A}} \nu\left( \pi\left( B(x,r) \cap R_a  \right) \right) \leq |\mathcal{A}| C \|\partial_u \pi\|_{\infty,\mathcal{R}}^{\delta_{up}} (2r)^{\delta_{up}} \leq \tilde{C} r^{\delta_{up}}.$$
using the previous lemma, and using the fact that $\text{diam}_u\left( \pi\left( B(x,r) \cap R_a\right) \right) \leq r \|\partial_u \pi\|_{\infty,\mathcal{R}}$.
\end{proof}

\subsection{ Large deviations }

We finish this preliminary section by recalling some large deviation results. There exists a large bibliography on the subject, see for example $\cite{Ki90}$. Large deviations in the context of Fourier decay for some 1-dimensional shift was used in $\cite{JS16}$ and in $\cite{SS20}$. A simple proof in the context of Julia sets using the pressure function can be found in the appendix of $\cite{Le21}$.

\begin{theorem}[\cite{LQZ03}]

Let $g : \Omega \rightarrow \mathbb{C}$ be any continuous map. Then, for all $\varepsilon>0$, there exists $n_0(\varepsilon)$ and $\delta_0(\varepsilon) > 0$ such that

$$ \forall n \geq n_0(\varepsilon), \ \mu \left( \left\{ x \in \Omega , \ \left| \frac{1}{n} \sum_{k=0}^{n-1} g(f^k(x)) - \int_{\Omega} g d\mu \right| \geq \varepsilon \right\} \right) \leq e^{-\delta_0(\varepsilon) n} $$

\end{theorem}

Notice that applying Theorem 2.11 to a $S$-constant function immediately gives the same statement for $(\nu,\widehat{\mathcal{U}},F)$. We apply it to some special cases.

\begin{definition}
Define $\tau_f := \ln |\partial_u f (x)| $  (so that $-\tau_f$ is the geometric potential) and $\tau_F := \ln |F'(x)|$.  Then, define the associated global Lyapunov exponents by:
$$ \lambda := \int_{\Omega} \tau_f d\mu > 0 \quad , \quad \Lambda := \int_{\mathcal{U}} \tau_F d\nu > 0  $$

Also, define the dimension of $\nu$ by the formula $$\delta := -\frac{1}{\Lambda} \int_{\mathcal{U}} \varphi d\nu > 0. $$
A consequence of Theorem 2.11 is the following:
\end{definition} 

\begin{corollary}[\cite{SS20}, \cite{Le21}]

For every $\varepsilon > 0$, there exists $n_1(\varepsilon)$ and $\delta_1(\varepsilon)>0$ such that

$$ \forall n \geq n_1(\varepsilon), \ \nu \left( \left\{ x \in \widehat{\mathcal{U}} \ , \ \left|\frac{1}{n}S_n \tau_F(x) -  \Lambda \right| \geq \varepsilon \text{ or } \left|\frac{S_n \varphi(x)}{S_n \tau_F(x)} + \delta \right| \geq \varepsilon \right\} \right) \leq e^{- \delta_1(\varepsilon) n} .$$

\end{corollary}

 Those quantitative results tells us that, \say{for most $x \in \Omega$}, $|\partial_u f^n(x)|$ have order of magnitude $\exp(\lambda n)$, and $|(F^n)'(x)|$ have order of magnitude $\exp(\Lambda n)$. Studying those two quantities is therefore central for us. An important (and natural) remark is the following.

\begin{lemma}
The global unstable Lyapunov exponents $\lambda$ and $\Lambda$ are equal. In other words,
$$ \int_{\Omega} \ln |\partial_u f| d\mu = \int_{\mathcal{U}} \ln |F'| d\nu. $$

\end{lemma}

\begin{proof}

The fact that $ \pi \circ f = F \circ \pi  $ implies that $$ \tau_f = \tau_F(\pi(x)) + \ln |\partial_u \pi(f(x))| - \ln |\partial_u \pi(x)| .$$

In particular, $ \tau_f $ and $\tau_F \circ \pi$ are $f$-cohomologous.
This implies the desired equality, by definition of the measure $\nu$, and by $f$-invariance of $\mu$. \end{proof}

\section{Computing some orders of magnitudes}

Starting now, we begin to do analysis. For this, it will be useful to work on connected sets. Hence, we make from now on the assumption that \textbf{$\Omega$ is an attractor}. A useful result in this case is the following. 

\begin{lemma}

Suppose that $\Omega$ is an attractor for $f$. Let $x \in \Omega$. Then $ W^u(x) \subset \Omega .$

In particular, in this setting, $(U_a)_{a \in \mathcal{A}}$ is a finite collection of compact connected smooth curves.

\end{lemma}

\begin{proof}
Let $x \in \Omega$ and $y \in W^u(x)$. Since $\Omega$ is an attractor, there exists $U$ an open neighborhood of $\Omega$ such that $\bigcap_{n \geq 0} f^n(U) = \Omega$. Since $d(f^{-n}(x),f^{-n}(y)) \underset{n \rightarrow + \infty}{\longrightarrow} 0$, there exists $N \geq 0$ so that $\tilde{y} := f^{-N}(y) \in U$, and for all $n \geq 0$, $f^{-n}(\tilde{y}) \in U$. We can then write $\tilde{y} = f^n(f^{-n}(\tilde{y})) \in f^n(U)$ for all $n \geq 0$. Hence $\tilde{y} \in \Omega$, and so $y \in \Omega$.
\end{proof}

This allows us to see holonomy maps as genuine $C^{1+\alpha}$ diffeomorphisms between curves, and to see $F$ as a smooth 1-dimensional map. It defines a piecewise expanding map in the sense of \cite{DV21}.   \\

The goal of this section is to use the large deviations to compute orders of magnitude for dynamical quantities. Once this preparatory step is done, we will be able to study oscillatory integrals involving $\mu$. First of all, we recall some useful symbolic notations.

\begin{itemize}
    \item Recall that $\mathcal{W}_n$ stands for the set of admissible words of length $n$. For $\mathbf{a}=a_1 \dots a_n a_{n+1} \in \mathcal{W}_{n+1}$, define $\mathbf{a}' := a_1 \dots a_{n} \in \mathcal{W}_n$.
    
    \item Let $\mathbf{a} \in \mathcal{W}_{n+1}$ and $\mathbf{b} \in \mathcal{W}_{m+1}$. Denote $\mathbf{a} \rightsquigarrow \mathbf{b} $ if $a_{n+1}=b_1$. In this case, $\mathbf{a}' \mathbf{b} \in \mathcal{W}_{n+m+1}$.
    
    \item Let $\mathbf{a} \in \mathcal{W}_{n+1}$. We denote by $b(\mathbf{a})$ the last letter of $\mathbf{a}$.
\end{itemize}

With those notations, we may rewrite the formula for the iterate of our transfer operator. For any continuous function $h : \mathcal{U} \longrightarrow \mathbb{C} $, we have

$$\forall b \in \mathcal{A}, \ \forall x \in U_b, \ \mathcal{L}_\varphi^n h(x) = \underset{\mathbf{a} \rightsquigarrow b}{\sum_{\mathbf{a} \in \mathcal{W}_{n+1}}} e^{S_n \varphi(g_{\mathbf{a}}(x))} h(g_{\mathbf{a}(x)}) = \underset{\mathbf{a} \rightsquigarrow b}{\sum_{\mathbf{a} \in \mathcal{W}_{n+1}}}  h(g_{\mathbf{a}}(x)) w_{\mathbf{a}}(x) .$$
where $$ w_{\mathbf{a}}(x) := e^{S_n\varphi(g_{\mathbf{a}}(x))} \simeq \nu(U_{\mathbf{a}}) .$$

Iterating $\mathcal{L}_\varphi^n$ again leads us to the formula
$$ \forall x \in U_b, \ \mathcal{L}_\varphi^{nk} h(x)  = \sum_{\mathbf{a}_1 \rightsquigarrow \dots \rightsquigarrow \mathbf{a_k} \rightsquigarrow b} h(g_{\mathbf{a}_1 ' \dots \mathbf{a}_{k-1} ' \mathbf{a}_k}(x) ) w_{\mathbf{a}_1 ' \dots \mathbf{a}_{k-1} ' \mathbf{a}_k}(x) .$$

\begin{definition}
Let, for small $\varepsilon>0$ and $n \geq 1$, $$ B_n(\varepsilon) := \left\{ x \in \mathcal{R} \ , \left|\frac{1}{n}S_n \tau_F(\pi(x)) -  \lambda \right| < \varepsilon \text{ and } \left|\frac{S_n \varphi(\pi(x))}{S_n \tau_F(\pi(x))} + \delta \right| < \varepsilon  \right\}  $$
$$\text{and} \ C_{n}(\varepsilon) :=  \left\{ x \in {\Omega} \ , \ \ \left| \frac{1}{n}\sum_{k=0}^{n-1} \tau_f(f^k(x)) - \lambda \right| < \varepsilon \right\} .$$

Finally, let $ A_{n}(\varepsilon) := B_n(\varepsilon) \cap C_n(\varepsilon) $. Then $\mu(A_{n}(\varepsilon)) \geq 1 - e^{-\delta_2(\varepsilon) n}$ for $n$ large enough depending on $\varepsilon$ and for some $\delta_2(\varepsilon)>0$, by Theorem 2.11 and Corollary 2.12.

\end{definition}

\begin{notations}
To simplify the reading, when two quantities dependent of $n$ satisfy $b_n \leq C a_n $ for some constant $C$, we denote it by $a_n \lesssim b_n$. If $a_n \lesssim b_n \lesssim c_n$, we denote it by $a_n \simeq b_n$. If there exists $c,C$ and $\beta$, independent of $n$ and $\varepsilon$, such that $ c e^{- \varepsilon \beta n} a_n \leq b_n \leq C e^{\varepsilon \beta n} a_n$, we denote it by $a_n \sim b_n$. Throughout the text $\beta$ will be allowed to change from line to line. It correspond to some positive constant. 
\end{notations}
Eventually, we will chose $\varepsilon$ small enough such that this exponentially growing term gets absorbed by the other leading terms, so we can neglect it.

\begin{lemma}
Let $\mathbf{a} \in \mathcal{W}_{n+1}(\varepsilon)$ be such that $U_{\mathbf{a}} \cap A_{n}(\varepsilon) \neq \emptyset$.  Then:

\begin{itemize}
    \item uniformly on $x \in U_{\mathbf{a}}$,  $|\partial_u f^{n}(x)| \sim e^{n \lambda}.$
    \item uniformly on $x \in U_{b(\mathbf{a})}, \  |g_\mathbf{a}'(x)| \sim e^{- n \lambda}$
    \item $\mathrm{diam}(\widehat{U}_\mathbf{a}), \ \mathrm{diam}(U_\mathbf{a})  \sim e^{- n \lambda} $
    \item uniformly on $x \in U_{\mathbf{a}}, \ w_{\mathbf{a}}(x) \sim e^{- \delta \lambda n} $
    \item $\nu(\widehat{U}_\textbf{a}) \sim e^{- \delta \lambda n}$
\end{itemize}

\end{lemma}

\begin{remark}

The definition of $\delta$ is chosen so that $ \mu(U_{\mathbf{a}}) \sim \text{diam}(U_\mathbf{a})^\delta $ when $U_{\mathbf{a}} \cap A_{n}(\varepsilon) \neq \emptyset$. Beware that the diameters are for the distance $d^u$.

\end{remark}

\begin{proof}

The first thing to notice is that, since any Hölder map have exponentially vanishing variations, Birkhoff sums are easy to control when $n$ becomes large. In particular, the estimates that are true on a piece of $U_{\mathbf{a}}$ will extends on $U_{\mathbf{a}}$. For example, let $x \in U_{\mathbf{a}}$, and let $y_{\mathbf{a}} \in U_{\mathbf{a}} \cap A_{n}(\varepsilon)$. Since $\tau_f$ has exponentially vanishing variations, we can write
$$ \left| \frac{1}{n}\sum_{k=0}^{n-1} \tau_f(f^k(x)) - \lambda \right| \leq \left| \frac{1}{n}\sum_{k=0}^{n-1} \tau_f(f^k(x)) - \frac{1}{n}\sum_{k=0}^{n-1} \tau_f(f^k(y_{\mathbf{a}})) \right| + \left| \frac{1}{n}\sum_{k=0}^{n-1} \tau_f(f^k(y_{\mathbf{a}})) - \lambda \right|. $$
Since $f^k(U_{\mathbf{a}}) \subset R_{a_{k+1} \dots a_n}$ is diffeomorphic to $ U_{a_{k+1} \dots a_n} $ by the mean of the holonomy map $\pi$ which have bounded differential, we have
$$ \frac{1}{n} \sum_{k=0}^{n-1} \left| \tau_f(f^k(x)) - \tau_f(f^k(y_{\mathbf{a}})) \right|  \leq \frac{1}{n} \sum_{k=0}^{n-1} |\tau_f|_\alpha\text{diam}\left( f^k(U_{\mathbf{a}}) \right)^\alpha  $$
$$ \lesssim \frac{1}{n} \sum_{k=0}^{n-1}\text{diam}(U_{a_{k+1} \dots a_{n}})^\alpha  \lesssim  \frac{1}{n} \sum_{k=0}^{n-1} \kappa^{\alpha k} \leq \frac{C}{n}  $$

for some constant $C$.
In particular we see that for $n$ large enough depending on $\varepsilon$, the estimate
$$ \forall x \in U_{\mathbf{a}}, \ \left| \frac{1}{n} \sum_{k=0}^{n-1} \tau_f(f^k(x)) - \lambda \right| < 2 \varepsilon $$
holds.
Similarly the estimates related to $B_n(\varepsilon)$ extends on $U_{\mathbf{a}}$ replacing $\varepsilon$ with $2 \varepsilon$ as long as $n$ is large enough. Once this is known, the estimates for $\partial_u(f^{n})$, $g_{\mathbf{a}}'$ and $w_{\mathbf{a}}$ are easy, and the estimates for the diameters is a direct consequence of our 1-dimensional setting and the mean value theorem. 
The estimates for $\nu$ follows from the Gibbs estimates. \end{proof}

\begin{definition}

For some fixed set $\omega \subset \mathcal{U}$ such that $\nu(\omega)>0$, define the set of ($\omega$-localized) $\varepsilon$-regular words by
$$ \mathcal{R}_{n+1}(\omega) := \left\{ \mathbf{a} \in \mathcal{W}_{n+1} \ | \ \omega \cap U_{\mathbf{a}} \cap A_{n}(\varepsilon) \neq \emptyset \right\}. $$
When no localization is asked, define $\mathcal{R}_{n+1} := \mathcal{R}_{n+1}(M)$. Then, define the set of ($\omega$-localized) $\varepsilon$-regular $k$-blocks by $$ \mathcal{R}_{n+1}^k(\omega) = \left\{ \mathbf{A}=\mathbf{a}_1' \dots \mathbf{a}_{k-1}' \mathbf{a}_k \in \mathcal{W}_{nk+1} \ | \ \forall i \geq 2, \ \mathbf{a}_i \in \mathcal{R}_{n+1}, \ \text{and} \ \mathbf{a}_1 \in \mathcal{R}_{n+1}(\omega) \right\} .$$ 
Finally, define the associated geometric points to be $$ R_{n+1}^k(\omega) := \bigcup_{\mathbf{A} \in \mathcal{R}_{n+1}^k(\omega) } U_{\mathbf{A}} .$$

The sets $\mathcal{R}_{n+1}^k$, $R_{n+1}^k$ denotes the regular blocks and associated geometric points when $\omega = M$.

\end{definition}

\begin{lemma}

For $\omega$ and $k \geq 1$ fixed, and for $ n $ large enough depending on $\varepsilon$, $$ \# \mathcal{R}_{n+1}^k(\omega) \sim e^{ k \delta \lambda n} .$$

\end{lemma}

\begin{proof}

First of all, notice that $ A_n(\varepsilon) \subset R_{n+1} $, so that $ \nu(\mathcal{U} \setminus R_{n+1}) \leq \nu( \mathcal{U} \setminus A_{n} ) \leq \nu(\mathcal{U} \setminus B_{n}) + \nu(\mathcal{U} \setminus C_{n})$. To control the term in $C_n$, we notice that, by the same argument than in Lemma 3.1:
$$ \forall x \in R_a, \  x \in C_n(\varepsilon) \Rightarrow [x,S_a] \subset C_n(2 \varepsilon) $$
as soon as $n$ is large enough. Hence, by the definition of the measure $\nu$:
$$ \nu(\mathcal{U} \setminus C_n) = \mu\left( \mathcal{R} \setminus [\mathcal{U} \cap C_n, \mathcal{S}] \right) \leq \mu(\mathcal{R} \setminus C_n(\varepsilon/2)) \lesssim e^{- \delta_1 n} .$$
This proves that $ \nu(R_{n+1}) \geq 1 - e^{-\delta_0 n} $ for $n$ large enough depending on $\varepsilon$ and for some $\delta_0(\varepsilon)>0$. \\
Since $R_{n+1}(\omega) = R_{n+1} \cap \omega$, it follows that $ \nu(\omega) - e^{-\delta_0 n} \leq \nu(R_{n+1}(\omega)) \leq \nu(\omega) $. \\

Then, define $\tilde{R}_{n+1} := \bigsqcup_{\mathbf{a} \in \mathcal{R}_{n+1}} \text{int}_{\mathcal{U}} \left( U_{\mathbf{a}} \right) $. From the point of view of the measure $\nu$, it is indistinguishable from $R_{n+1}$. First, we prove that
$$ \omega \cap \bigcap_{i=0}^{k-1} F^{-n i} \left( \tilde{R}_{n+1} \right) \subset R_{n+1}^k(\omega) .$$
Let $x \in \omega \cap \bigcap_{i=0}^{k-1} F^{-n i} \left( \tilde{R}_{n+1} \right) $. Since there exists $\mathbf{A}=\mathbf{a}_1' \dots \mathbf{a}_{k-1}' \mathbf{a} _k\in \mathcal{W}_{kn+1}$ such that $x \in U_{\mathbf{A}}$, we see that for any $i$ we can write $F^{ni}(x) \in U_{\mathbf{a}_{1+i}' \dots \mathbf{a}_k} \cap \tilde{R}_{n+1} $. 
So there exists $\mathbf{b}_{i+1} \in \mathcal{R}_{n+1}$ such that $ U_{\mathbf{a}_{1+i}' \dots \mathbf{a}_k} \cap \text{int}_{\mathcal{U}} U_{\mathbf{b}_{i+1}} \neq \emptyset$. Then $\mathbf{b}_{i+1} = \mathbf{a}_{i+1}$, for all $i$, which implies that $\mathbf{A} \in \mathcal{R}_{n+1}^k$. Moreover, since $\emptyset \neq \omega \cap U_{\mathbf{A}} \subset \omega \cap U_{\mathbf{a}_1}$, $\mathbf{a}_1 \in R_{n+1}(\omega)$, and so $\mathbf{A} \in \mathcal{R}_{n+1}^k(\omega)$. 
Now that the inclusion is proved, we see that $$ \nu\left( \omega \setminus R_{n+1}^k(\omega) \right) \leq \nu(\omega \setminus R_{n+1}(\omega)) + \sum_{i=1}^{k-1} \nu\left( F^{-n i} \left( \mathcal{U} \setminus \tilde{R}_{n+1} \right) \right)  = \nu( \omega \setminus R_{n+1}(\omega)) + (k-1) \nu( \mathcal{U} \setminus R_{n+1} ) ,$$

and so $ \nu( R_{n+1}^k(\omega) ) \geq \nu(\omega) - k e^{-\delta_1(\varepsilon) n} $ for $n$ large enough depending on $\varepsilon$.

Now we may prove the cardinality estimate: since $ \nu(R_{n+1}^k(\omega)) = \sum_{\mathbf{A} \in R_{n+1}^k(\omega)} \nu(U_{\mathbf{A}}) $, we have

$$ \# \left(\mathcal{R}_{n+1}^k(\omega)\right) e^{- \varepsilon \beta n} e^{- \delta \lambda k n} \lesssim \nu\left(R_{n+1}^k(\omega)\right) \lesssim \# \left(\mathcal{R}_{n+1}^k(\omega)\right) e^{\varepsilon \beta n} e^{- \delta \lambda k n} $$
and so
$$  e^{-\varepsilon \beta n} e^{\delta \lambda k n} \left( \nu(\omega) - k e^{- \delta_1(\varepsilon) n} \right) \lesssim \# \mathcal{R}_{n+1}^k(\omega) \lesssim  e^{\varepsilon \beta n} e^{\delta \lambda k n} .$$ \end{proof}

\section{Reduction to sums of exponential}

We can finally start the proof of our main theorem. The goal is to follow the strategy developed in $\cite{SS20}$ and $\cite{Le21}$ (which generalized the original method developed in \cite{BD17} and \cite{LNP19}). To do so, we will reduce our oscillatory integral over $\Omega$ with respect to $\mu$ to an  oscillatory integral over $\mathcal{U}$ with respect to $\nu$. The fact that $\nu$ is invariant by a transfer operator related to an expanding map will allows us to get the same final reduction as in those two papers: everything will boils down to a sum of exponential that we will be able to control thanks to a combinatorial theorem of Bourgain. \\

Before going on, recall the setting: we have fixed a $C^{2 + \alpha}$ Axiom A diffeomorphism $f:M \rightarrow M$ and an attractor $\Omega$ on which $f$ has codimension 1 stable lamination and that satisfies our generic nonlinearity condition (NL). The bunching condition (B) will only be used in section 6. We have fixed a Hölder potential $\psi : \Omega \rightarrow \mathbb{R}$ and its associated equilibrium state $\mu$. We denote by $(F,\mathcal{U},\nu)$ the expanding dynamical system in factor defined in section 2. The measure $\nu$ is invariant by a transfer operator $\mathcal{L}_\varphi$, where $\varphi$ is some normalized and $\alpha$-Hölder potential. The Hölder exponent $\alpha$ is fixed for the rest of the paper. \\

Six quantities will be at play: $\xi$, $K$, $n$, $k$, $\varepsilon_0$ and $\varepsilon$. We will think of $k$, $\varepsilon_0$ and $\varepsilon$ as being fixed. The constant $k$ is fixed using Corollary 5.3. The constant $\varepsilon_0>0$ will be fixed at at the end of the paper, in Lemma 6.5. The constant $\varepsilon>0$ is chosen at the very end of the proof to be small in front of every other constant that might appear. 
The only variable is $\xi$. We relate it to $n$ and $K$ by the formulas $$ n := \left\lfloor \frac{\ln |\xi|}{(2k+1) \lambda + \varepsilon_0} \right\rfloor \quad \text{and} \quad K := \left\lfloor \frac{((2k+1) \lambda + 2 \varepsilon_0)n}{\alpha |\ln \kappa|} \right\rfloor. $$

\begin{definition}

Let $\chi \in C^\alpha(M,\mathbb{C})$ (with support of positive measure). Let $\phi \in C^{1+\alpha}(M,\mathbb{R})$. Suppose that there exists a constant $C_{\phi,\chi}>1$ such that
\begin{itemize}
    \item $ \| \chi \|_{C^\alpha} + \|\phi\|_{C^{1+\alpha}} < C_{\phi,\chi} $
    \item  $\forall x \in \text{supp} \ \chi , \ |\partial_u \phi(x)| > C_{\phi,\chi}^{-1}$. 
\end{itemize}
Then define, for $\xi \in \mathbb{R}$:
$$ \widehat{\mu}(\xi) := \int_{M} e^{i \xi \phi} \chi \ d\mu . $$

\end{definition}

Since the condition on $\partial_u \phi$ is open, it holds on a small open neighborhood $\tilde{\omega} \supset \text{supp} \chi$. Choosing a small enough Markov partition allows us to choose $\tilde{\omega}$ as a nonempty union of rectangles $R_a$. It follows that $\omega := \tilde{\omega} \cap \mathcal{U}$ have positive measure, and hence we can use the results of Lemma 3.2.

\begin{remark}
The constant $C$ in Theorem 1.2 will depends only on the riemannian manifold $M$, the diffeomorphism $f$, the equilibrium state $\mu$, the Hölder exponent $\alpha$ and the constant $C_{\phi,\chi} > 1$.
\end{remark}

\begin{notations}
We recall a final set of notations, inspired from \cite{BD17}. For a fixed $n$ and $k$, denote:

\begin{itemize}
    \item $\textbf{A}=(\textbf{a}_0, \dots, \textbf{a}_k) \in \mathcal{W}_{n+1}^{k+1} \ , \ \textbf{B}=(\textbf{b}_1, \dots, \textbf{b}_k) \in \mathcal{W}_{n+1}^{k} $.
    \item We write $\textbf{A} \leftrightarrow \textbf{B}$ iff $\textbf{a}_{j-1} \rightsquigarrow \textbf{b}_j \rightsquigarrow \textbf{a}_j$ for all $j=1,\dots k$.
    \item If $\textbf{A} \leftrightarrow \textbf{B}$, then we define the words $\textbf{A} * \textbf{B} := \textbf{a}_0' \textbf{b}_1' \textbf{a}_1' \textbf{b}_2' \dots \textbf{a}_{k-1}' \textbf{b}_k' \textbf{a}_k$  \\ and  $\textbf{A} \# \textbf{B} :=  \textbf{a}_0' \textbf{b}_1' \textbf{a}_1' \textbf{b}_2' \dots \textbf{a}_{k-1}' \textbf{b}_k$.
    \item Denote by $b(\textbf{A}) \in \mathcal{A}$ the last letter of $\textbf{a}_k$.
    
\end{itemize}

\end{notations}

We prove the following reduction.

\begin{proposition}

Define $J_n := \{ e^{\varepsilon_0 n/2} \leq |\eta| \leq  e^{2 \varepsilon_0 n} \}$ and $$  \zeta_{j,\mathbf{A}}(\mathbf{b}) := e^{2 \lambda n} |g_{\mathbf{a}_{j-1}' \mathbf{b}}'(x_{\mathbf{a}_j})| .$$
There exists some constant $\beta>1$ such that for $n$ large enough depending on $\varepsilon$:
$$ |\widehat{\mu}(\xi)|^2 \lesssim e^{ \varepsilon \beta n} e^{-\lambda \delta (2k+1) n} \sum_{\mathbf{A} \in \mathcal{R}_{n+1}^{k+1}(\omega)} \sup_{\eta \in J_n} \Bigg{|} \underset{\mathbf{A} \leftrightarrow \mathbf{B}}{\sum_{\mathbf{B} \in \mathcal{R}_{n+1}^k}} e^{i \eta \zeta_{1,\mathbf{A}}(\mathbf{b}_1) \dots \zeta_{k,\mathbf{A}}(\mathbf{b}_k) } \Bigg| $$
$$ \quad \quad  \quad \quad  \quad \quad  \quad \quad + e^{-  \varepsilon_0 n} +  e^{-\delta_1(\varepsilon) n} + e^{\varepsilon \beta n} \left( e^{- \lambda \alpha n} +  \kappa^{\alpha n} + e^{-(\alpha \lambda-\varepsilon_0) n} + e^{- \varepsilon_0 \delta_{up}n/2} \right).$$
\end{proposition}

Once proposition 4.1 is established, if we manage to prove that the sum of exponentials enjoys exponential decay in $n$, then choosing $\varepsilon$ small enough will allow us to see that $|\widehat{\mu}(\xi)|^2$ enjoys polynomial decay in $\xi$, and our main theorem will be proved. We prove Proposition 4.1 through a succession of lemmas.

\begin{lemma}

For $n$ large enough there exists $C>0$ that depends only on $C_{\psi,\chi}$ such that

$$ \left| \int_{\Omega} e^{i \xi \phi} \chi d\mu - \int_{\mathcal{U}} e^{i \xi \phi \circ f^K} \chi \circ f^K d \nu \right| \leq C e^{- \varepsilon_0 n} . $$

\end{lemma}

\begin{proof}

Let $h(x) := e^{i \xi \phi(x)} \chi(x) $. Then, since $\mu$ is $f$ invariant, $$ \widehat{\mu} = \int_\Omega h \circ f^K d\mu .$$

We then check that $h \circ f^K$ is close to a $S$-constant function as $K$ grows larger.
First of all, notice that since $\phi$ is $C^1$, $e^{i \xi \phi}$ also is. In particular it is Lipschitz with constant $|\xi| C_{\phi, \chi}$. Since $M$ is a closed riemannian manifold, it follows that $e^{i \xi \phi}$ is $\alpha$-Hölder on $M$ with constant $|\xi| C_{\phi,\chi} \text{diam}(M)^{1-\alpha}$. Since $\chi$ is $\alpha$-Holder too, with constant $C_{\phi,\chi}$, the product $h$ is also locally $\alpha$-Hölder, with constant $ C_{\phi,\chi} + C_{\phi,\chi}^2 |\xi| \text{diam}(M)^{1-\alpha}  $. \\

Let $x \in \mathcal{R}_a$. By definition, $\pi(x) \in U_a$ is in $W^s(x)$. Then $$ |h(f^K(x)) - h(f^K(\pi(x))) | \leq |h|_{\alpha} d( f^K(x),f^K(\pi(x)) )^\alpha \leq |h|_\alpha \kappa^{\alpha K}$$
$$ \leq \left( C_{\phi,\chi} + C_{\phi,\chi}^2 |\xi| \text{diam}(M)^{1-\alpha}  \right) \kappa^{\alpha K} .$$

Since $|\xi| \kappa^{\alpha K} \lesssim e^{- \varepsilon_0 n}$, there exists a constant $C>1$ that depends only on $C_{\psi,\chi}$ such that 
$$ \| h \circ f^K - h \circ f^K \circ \pi \|_{\infty,\mathcal{R}} \leq C e^{- \varepsilon_0 n} .$$

The desired estimates follows from the definition of $\nu$. \end{proof}

Now we are ready to adapt the existing strategy for one dimensional expanding maps. From this point, we will follow $\cite{SS20}$ and $\cite{Le21}$. Notice however that we use an additional word $\mathbf{C}$ to cancel the distortions induced by $f^K$.

\begin{lemma}

$$ \left| \int_{\mathcal{U}} e^{i \xi \phi \circ f^K} \chi \circ f^K d \nu \right|^2 \lesssim \quad \quad \quad \quad \quad \quad \quad \quad \quad \quad \quad \quad \quad \quad \quad \quad \quad \quad \quad \quad \quad \quad \quad \quad \quad \quad \quad \quad \quad \quad $$ $$\quad \quad \quad  \Bigg{|} \sum_{\mathbf{C} \in \mathcal{R}_{K+1}} \underset{\mathbf{C} \rightsquigarrow \mathbf{A} \leftrightarrow \mathbf{B}}{ \underset{ \mathbf{B} \in \mathcal{R}_{n+1}^k}{ \underset{\mathbf{A} \in \mathcal{R}_{n+1}^{k+1} }{ \sum}} } \int_{U_{b(\mathbf{A})}} e^{i \xi \phi( f^K g_{\mathbf{C}'(\mathbf{A} * \mathbf{B})}(x))}  \chi(f^K g_{\mathbf{C}'(\mathbf{A} * \mathbf{B})}(x))  w_{\mathbf{C}'(\mathbf{A} * \mathbf{B})}(x)d\nu(x) \Bigg{|}^2 + e^{- \delta_1(\varepsilon) n } $$

\end{lemma}

\begin{proof}

Since $\nu$ is invariant by $\mathcal{L}_\varphi$, we can write

$$ \int_{\mathcal{U}} h \circ f^K d \nu = \sum_{\mathbf{C} \in \mathcal{W}_{K+1}}  \int_{U_{b(\mathbf{C})}} h\left(f^K(g_{ \mathbf{C}}(x))\right) w_{\mathbf{C}}(x) d\nu(x) $$

If we look at the part of the sum where $\mathbf{C}$ is not a regular word, we get by the Gibbs estimates:

$$ \Bigg| \sum_{\mathbf{C} \notin \mathcal{R}_{K+1}} \int_{U_{b(\mathbf{A})}} h\left(f^K(g_{\mathbf{C}}(x))\right) w_{\mathbf{C}}(x) d\nu(x) \Bigg| \lesssim \sum_{\mathbf{B} \notin \mathcal{R}_K} \nu(U_{\mathbf{C}}) \lesssim \nu(\mathcal{U} \setminus R_{K}) $$
which decays exponentially in $n$ by Lemma 3.2.
Then, we iterate our transfer operator again on the main sum, to get the following term:
$$  \sum_{\mathbf{C} \in \mathcal{R}_{K+1}} \underset{\mathbf{C} \rightsquigarrow \mathbf{A} \leftrightarrow \mathbf{B}}{ \underset{ \mathbf{B} \in \mathcal{W}_{n+1}^k}{ \underset{\mathbf{A} \in \mathcal{W}_{n+1}^{k+1} }{ \sum}} } \int_{U_{b(\mathbf{A})}} e^{i \xi \phi( f^K g_{\mathbf{C}'(\mathbf{A} * \mathbf{B})}(x))}  \chi(f^K g_{\mathbf{C}'(\mathbf{A} * \mathbf{B})}(x))  w_{\mathbf{C}'(\mathbf{A} * \mathbf{B})}(x)d\nu(x) . $$

Looking at the part of the sum where words are not regular again, we see by the Gibbs estimates again that

$$ \Bigg| \sum_{\mathbf{C} \in \mathcal{R}_{K+1}} \underset{\text{or }\mathbf{B} \notin \mathcal{R}_{n+1}^k}{\underset{\mathbf{A} \notin \mathcal{R}_{n+1}^{k+1} }{\sum_{\mathbf{C} \rightsquigarrow \mathbf{A} \leftrightarrow \mathbf{B}}}} \int_{U_{b(\mathbf{A})}} h\left(f^K(g_{\mathbf{C}'(\mathbf{A} * \mathbf{B})}(x))\right) w_{\mathbf{C}' (\mathbf{A} * \mathbf{B})}(x) d\nu(x) \Bigg| \quad \quad  \quad \quad  \quad \quad  \quad \quad  \quad \quad $$ $$  \quad \quad  \quad \quad  \quad \quad  \quad \quad  \lesssim \sum_{\mathbf{C} \in \mathcal{R}_{K+1}} \underset{\text{or }\mathbf{B} \notin \mathcal{R}_{n+1}^k}{\underset{\mathbf{A} \notin \mathcal{R}_{n+1}^{k+1} }{\sum_{\mathbf{C} \rightsquigarrow \mathbf{A} \leftrightarrow \mathbf{B}}}} \nu(U_{\mathbf{C}}) \nu(U_{\mathbf{A} * \mathbf{B}}) \lesssim \nu(\mathcal{U} \setminus R_{n}^{2k+1}), $$
and the desired estimate follows from Lemma 3.2. \end{proof}

\begin{lemma}

Define $ \chi_{\mathbf{C}}(\mathbf{a}_0) := \chi(f^K g_{\mathbf{C}}(x_{\mathbf{a}_0}))$. There exists some constant $\beta>0$ such that, for $n$ large enough:

$$  \Bigg{|}\sum_{\mathbf{C} \in \mathcal{R}_{K+1}} \underset{\mathbf{C} \rightsquigarrow \mathbf{A} \leftrightarrow \mathbf{B}}{ \underset{ \mathbf{B} \in \mathcal{R}_{n+1}^k}{ \underset{\mathbf{A} \in \mathcal{R}_{n+1}^{k+1} }{ \sum}} } \int_{U_{b(\mathbf{A})}} e^{i \xi \phi( f^K g_{\mathbf{C}'(\mathbf{A} * \mathbf{B})}(x))} w_{\mathbf{C}'(\mathbf{A} * \mathbf{B})}(x) \chi(f^K g_{\mathbf{C}'(\mathbf{A} * \mathbf{B})}(x)) d\nu(x) \Bigg{|}^2  \quad \quad \quad \quad \quad \quad \quad $$ $$ \quad \quad \quad \quad  \lesssim  \Bigg{|} \sum_{\mathbf{C} \in \mathcal{R}_{K+1}} \underset{\mathbf{C} \rightsquigarrow \mathbf{A} \leftrightarrow \mathbf{B}}{ \underset{ \mathbf{B} \in \mathcal{R}_{n+1}^k}{ \underset{\mathbf{A} \in \mathcal{R}_{n+1}^{k+1}(\omega) }{ \sum}} } \chi_{\mathbf{C}}(\mathbf{a}_0) \int_{U_{b(\mathbf{A})}} e^{i \xi \phi( f^K g_{\mathbf{C}'(\mathbf{A} * \mathbf{B})}(x)))} w_{\mathbf{C}'(\mathbf{A} * \mathbf{B})}(x) d\nu(x) \Bigg{|}^2 + e^{\varepsilon \beta n} e^{- \lambda \alpha n}. $$

\end{lemma}

\begin{proof}

Our first move is to get rid of terms in the sum where $\chi(f^K g_{\mathbf{C}} g_{\mathbf{A} * \mathbf{B}}) =0 $. To this end, notice that $f^K g_{\mathbf{C}}$ is an holonomy map that sends $U_{b(\mathbf{C})}$ in $R_{b(\mathbf{C})}$. It doesn't really play a role for this question. The only word that matter here is $\mathbf{a}_0$ : if it isn't in $R_{n+1}(\omega)$, then by definition it implies that $\omega \cap U_{\mathbf{a}_0} = \emptyset$. It follows that $\tilde{\omega} \cap R_{\mathbf{a}_0} = \emptyset$, and so $\chi(f^K g_{\mathbf{C}} g_{\mathbf{A} * \mathbf{B}}) = 0$. \\

Hence our main term is equal to the same sum, but where we have restricted $\mathbf{A}$ in $\mathcal{R}_{n+1}(\omega)$. \\

Next, denote $ \chi_{\mathbf{C}}(\mathbf{a}_0) := \chi(f^K g_{\mathbf{C}}(x_{\mathbf{a}_0})). $ The orders of magnitude of Lemma 3.1 and Remark 3.3 combined gives us
$$ \left| \chi(f^K g_{\mathbf{C}} g_{\mathbf{A}*\mathbf{B}}(x)) -  \chi_{\mathbf{C}}(\mathbf{a}_0)  \right| \leq C_{\phi,\chi} d\left(f^K g_{\mathbf{C}} g_{\mathbf{A}*\mathbf{B}}(x),f^K g_{\mathbf{C}} (x_{\mathbf{a}_0} ) \right)^\alpha $$ 
$$ \leq C_{\phi,\chi} \|\partial_u (f^K)\|_{\infty,U_{\mathbf{C}}}^\alpha \| g_{\mathbf{C}}' \|_{\infty,U_{b(\mathbf{C})}}^\alpha  \text{diam}(U_{\mathbf{a}_0 })^\alpha  \lesssim e^{\varepsilon \beta n} e^{- \lambda \alpha n} .$$

Hence, by the Gibbs estimates 
$$ { \Bigg{|} {\sum_{\mathbf{A},\mathbf{B},\mathbf{C}}} \int_{U_{b(\mathbf{A})}} e^{i \xi \phi( f^K g_{\mathbf{C}'(\mathbf{A} * \mathbf{B})} )} \chi( f^K g_{\mathbf{C}'(\mathbf{A} * \mathbf{B})})  w_{\mathbf{C}'(\mathbf{A} * \mathbf{B})} d\nu  -   \sum_{\mathbf{A},\mathbf{B},\mathbf{C}} \chi_{\mathbf{C}}(\mathbf{a}_0)  \int_{U_{b(\mathbf{A})}} e^{i \xi \phi(f^K g_{\mathbf{C}'(\mathbf{A} * \mathbf{B})})} w_{\mathbf{C}'(\mathbf{A} * \mathbf{B})} d\nu \Bigg{|} }$$
$$ \lesssim e^{\varepsilon \beta n} e^{- \lambda \alpha n}  \sum_{\mathbf{A},\mathbf{B},\mathbf{C}} \nu(U_{\mathbf{C}'(\mathbf{A} * \mathbf{B}) }) \lesssim e^{\varepsilon \beta n} e^{- \lambda \alpha n} .$$

\end{proof}

\begin{lemma}
There exists some constant $\beta>0$ such that, for $n$ large enough:
$$ \Bigg{|} \sum_{\mathbf{C} \in \mathcal{R}_{K+1}} \underset{\mathbf{C} \rightsquigarrow \mathbf{A} \leftrightarrow \mathbf{B}}{ \underset{ \mathbf{B} \in \mathcal{R}_{n+1}^k}{ \underset{\mathbf{A} \in \mathcal{R}_{n+1}^{k+1}(\omega) }{ \sum}} } \chi_{\mathbf{C}}(\mathbf{a}_0) \int_{U_{b(\mathbf{A})}} e^{i \xi \phi( f^K g_{\mathbf{C}'(\mathbf{A} * \mathbf{B})}(x)))} w_{\mathbf{C}'(\mathbf{A} * \mathbf{B})}(x) d\nu(x) \Bigg{|}^2 \quad \quad \quad \quad \quad \quad \quad  $$
$$ \quad \quad \quad \quad \lesssim  e^{-(2k-1) \lambda \delta n} e^{- \lambda \delta K} \sum_{\mathbf{C} \in \mathcal{R}_{K+1}} \underset{\mathbf{C} \rightsquigarrow \mathbf{A} \leftrightarrow \mathbf{B}}{ \underset{ \mathbf{B} \in \mathcal{R}_{n+1}^k}{ \underset{\mathbf{A} \in \mathcal{R}_{n+1}^{k+1}(\omega) }{ \sum}} } \Bigg{|} \int_{U_{b(\mathbf{A})}}  e^{i \xi \phi( f^K( g_{\mathbf{C}'(\mathbf{A} * \mathbf{B})}(x)))} w_{\mathbf{a}_k}(x) d\nu(x) \Bigg{|}^2 + e^{\varepsilon \beta n} \kappa^{\alpha n}. $$
\end{lemma}

\begin{proof}

Notice that $w_{\mathbf{C}'(\mathbf{A} * \mathbf{B})}(x)$ and $w_{\mathbf{a}_k}(x)$ are related by
$$ w_{\mathbf{C}'(\mathbf{A} * \mathbf{B})}(x) = w_{\mathbf{C}'(\mathbf{A} \# \mathbf{B})}(g_{\mathbf{a}_k}(x)) w_{\mathbf{a}_k}(x) .$$

Moreover:
$$  \frac{w_{\mathbf{C}'(\mathbf{A} \# \mathbf{B})}(g_{\mathbf{a}_k}(x))}{w_{\mathbf{C}'(\mathbf{A} \# \mathbf{B})}(x_{\mathbf{a}_k})} = \exp \left( S_{K+2nk}\varphi(g_{\mathbf{C}'(\mathbf{A} \# \mathbf{B})}(g_{\mathbf{a}_k}(x))) - S_{K+2kn}\varphi( g_{\mathbf{C}'(\mathbf{A} \# \mathbf{B})}(x_{\mathbf{a}_k})) \right) $$
with 
$$ \left| S_{K+2nk}\varphi(g_{\mathbf{C}'(\mathbf{A} \# \mathbf{B})}(g_{\mathbf{a}_k}(x)) - S_{K+2kn}\varphi( g_{\mathbf{C}'(\mathbf{A} \# \mathbf{B})}(x_{\mathbf{a}_k}))  \right| \lesssim \sum_{j=0}^{K+2nk-1} \kappa^{\alpha(K+n(2k+1) - j )} \lesssim \kappa^{\alpha n} $$
since $\varphi$ is $\alpha$-Hölder. Hence, there exists some constant $C>0$ such that
$$ e^{-C \kappa^{\alpha n}} w_{\mathbf{C}'(\mathbf{A} \# \mathbf{B})}(x_{\mathbf{a}_k}) \leq w_{\mathbf{C}'(\mathbf{A} \# \mathbf{B})}(g_{\mathbf{a}_k}(x)) \leq e^{C \kappa^{\alpha n}} w_{\mathbf{C}'(\mathbf{A} \# \mathbf{B})}(x_{\mathbf{a}_k}) $$
Which gives:
$$ \left| w_{\mathbf{C}'(\mathbf{A} \# \mathbf{B})}(g_{\mathbf{a}_k}(x)) -  w_{\mathbf{C}'(\mathbf{A} \# \mathbf{B})}(x_{\mathbf{a}_k})  \right| \leq \max\left| e^{\pm C \kappa^{\alpha n}} -1 \right|  w_{\mathbf{C}'(\mathbf{A} \# \mathbf{B})}(x_{\mathbf{a}_k}) \lesssim \kappa^{\alpha n} w_{\mathbf{C}'(\mathbf{A} \# \mathbf{B})}(x_{\mathbf{a}_k}) .$$
Hence
$$ \Bigg{|}  \sum_{\mathbf{C} \in \mathcal{R}_{K+1}} \underset{\mathbf{C} \rightsquigarrow \mathbf{A} \leftrightarrow \mathbf{B}}{ \underset{ \mathbf{B} \in \mathcal{R}_{n+1}^k}{ \underset{\mathbf{A} \in \mathcal{R}_{n+1}^{k+1}(\omega) }{ \sum}} } \chi_{\mathbf{C}}(\mathbf{a}_0) \int_{U_{b(\mathbf{A})}} e^{i \xi \phi(f^K g_{\mathbf{C}'(\mathbf{A} * \mathbf{B})}(x))} \left(w_{\mathbf{C}'(\mathbf{A} * \mathbf{B})}(x) - w_{\mathbf{C}'(\mathbf{A} \# \mathbf{B})}(x_{\mathbf{a}_k}) w_{\mathbf{a}_k}(x) \right) d\nu(x) \Bigg{|}   $$

$$ \lesssim \kappa^{\alpha n}  \sum_{\mathbf{A},\mathbf{B},\mathbf{C}}\int_{U_{b(\mathbf{A})}}  w_{\mathbf{C}'(\mathbf{A} \# \mathbf{B})}(x_{\mathbf{a}_k}) w_{\mathbf{a}_k}(x) d\nu(x) \ \lesssim e^{\varepsilon \beta n} \kappa^{\alpha n} .$$
by the Gibbs estimates.
Moreover, by Cauchy-Schwartz and by the orders of magnitude of Lemma 3.1,
$$ \Bigg{|}  \sum_{\mathbf{C} \in \mathcal{R}_{K+1}} \underset{\mathbf{C} \rightsquigarrow \mathbf{A} \leftrightarrow \mathbf{B}}{ \underset{ \mathbf{B} \in \mathcal{R}_{n+1}^k}{ \underset{\mathbf{A} \in \mathcal{R}_{n+1}^{k+1}(\omega) }{ \sum}} } w_{\mathbf{C}'(\mathbf{A} \# \mathbf{B})}(x_{\mathbf{a}_k})  \chi_{\mathbf{C}}(\mathbf{a}_0) \int_{U_{b(\mathbf{A})}} e^{i \xi \phi(f^K g_{\mathbf{C}'(\mathbf{A} * \mathbf{B})}(x) ) }   w_{\mathbf{a}_k}(x) d\nu(x) \Bigg{|}^2   $$
$$  \lesssim e^{\varepsilon \beta n} e^{-\lambda \delta (2k-1) n} e^{- \lambda \delta K} \sum_{\mathbf{C} \in \mathcal{R}_{K+1}} \underset{\mathbf{C} \rightsquigarrow \mathbf{A} \leftrightarrow \mathbf{B}}{ \underset{ \mathbf{B} \in \mathcal{R}_{n+1}^k}{ \underset{\mathbf{A} \in \mathcal{R}_{n+1}^{k+1}(\omega) }{ \sum}} }\left| \int_{U_{b(\mathbf{A})}} e^{i \xi \phi(f^K g_{\mathbf{C}' (\mathbf{A} * \mathbf{B})}(x) ) }   w_{\mathbf{a}_k}(x) d\nu(x) \right|^2 ,$$

where one could increase $\beta$ if necessary.\end{proof}

\begin{lemma}
Define $$ \Delta_{\mathbf{A},\mathbf{B},\mathbf{C}}(x,y) :=  \phi( f^K g_{\mathbf{C}'(\mathbf{A} * \mathbf{B})}(x)) - \phi( f^K g_{\mathbf{C}'(\mathbf{A} * \mathbf{B})}(y)). $$
There exists some constant $\beta>0$ such that, for $n$ large enough:
$$ e^{-(2k-1) \lambda \delta n} e^{- \lambda \delta K} \sum_{\mathbf{C} \in \mathcal{R}_{K+1}} \underset{\mathbf{C} \rightsquigarrow \mathbf{A} \leftrightarrow \mathbf{B}}{ \underset{ \mathbf{B} \in \mathcal{R}_{n+1}^k}{ \underset{\mathbf{A} \in \mathcal{R}_{n+1}^{k+1}(\omega) }{ \sum}} } \Bigg{|} \int_{U_{b(\mathbf{A})}}  e^{i \xi \phi( f^K g_{\mathbf{C}'(\mathbf{A} * \mathbf{B})}(x))} w_{\mathbf{a}_k}(x) d\nu(x) \Bigg{|}^2 $$
$$ \lesssim e^{\varepsilon \beta n } e^{-\lambda \delta (2k+1) n} e^{- \lambda \delta K} \underset{\mathbf{C} \rightsquigarrow \mathbf{A}}{\underset{\mathbf{A} \in \mathcal{R}_{n+1}^{k+1}(\omega)}{\sum_{\mathbf{C} \in \mathcal{R}_{K+1}}}} \iint_{U_{b(\mathbf{A})}^2 }  \Bigg{|} \underset{\mathbf{A} \leftrightarrow \mathbf{B}}{\sum_{\mathbf{B} \in \mathcal{R}_{n+1}^k}} e^{i \xi \left|\Delta_{\mathbf{A},\mathbf{B},\mathbf{C}}\right|(x,y) } \Bigg| d\nu(x) d\nu(y)  .$$

\end{lemma}

\begin{proof}

We first open up the modulus squared:

$$  \underset{\mathbf{B} \in \mathcal{R}_{n+1}^k}{\underset{\mathbf{A} \in \mathcal{R}_{n+1}^{k+1}(\omega) }{\sum_{\mathbf{A} \leftrightarrow \mathbf{B}}}}  \Bigg{|} \int_{U_{b(\mathbf{A})}}  e^{i \xi \phi( f^K g_{\mathbf{C}'(\mathbf{A} * \mathbf{B})}(x))} w_{\mathbf{a}_k}(x) d\nu(x) \Bigg{|}^2  $$ $$= \underset{\mathbf{B} \in \mathcal{R}_{n+1}^k}{\underset{\mathbf{A} \in \mathcal{R}_{n+1}^{k+1}(\omega) }{\sum_{\mathbf{A} \leftrightarrow \mathbf{B}}}} \iint_{U_{b(\mathbf{A})}^2} w_{\mathbf{a}_k}(x) w_{\mathbf{a}_k}(y) e^{i \xi \Delta_{\mathbf{A},\mathbf{B},\mathbf{C}}(x,y)  }  d\nu(x) d\nu(y).  $$
Since this quantity is real, we get:
$$ = \underset{\mathbf{B} \in \mathcal{R}_{n+1}^k}{\underset{\mathbf{A} \in \mathcal{R}_{n+1}^{k+1}(\omega) }{\sum_{\mathbf{A} \leftrightarrow \mathbf{B}}}} \iint_{U_{b(\mathbf{A})}^2} w_{\mathbf{a}_k}(x) w_{\mathbf{a}_k}(y) \cos({ \xi \Delta_{\mathbf{A},\mathbf{B},\mathbf{C}}(x,y)  })  d\nu(x) d\nu(y) $$

$$ = \sum_{\mathbf{A} \in \mathcal{R}_{n+1}^{k+1}(\omega)} \iint_{U_{b(\mathbf{A})}^2 }  w_{\mathbf{a}_k}(x) w_{\mathbf{a}_k}(y) \underset{\mathbf{A} \leftrightarrow \mathbf{B}}{\sum_{\mathbf{B} \in \mathcal{R}_{n+1}^k}}  \cos\left( \xi  |\Delta_{\mathbf{A},\mathbf{B},\mathbf{C}}|(x,y) \right) d\nu(x) d\nu(y),$$

and then we conclude using the triangle inequality and the estimates of section 3 as follow:

$$ \lesssim e^{\varepsilon \beta n } e^{- 2 \lambda \delta n} \sum_{\mathbf{A} \in \mathcal{R}_{n+1}^{k+1}(\omega)} \iint_{U_{b(\mathbf{A})}^2 }  \Bigg{|} \underset{\mathbf{A} \leftrightarrow \mathbf{B}}{\sum_{\mathbf{B} \in \mathcal{R}_{n+1}^k}} e^{i \xi \left|\Delta_{\mathbf{A},\mathbf{B},\mathbf{C}}\right|(x,y) } \Bigg| d\nu(x) d\nu(y) . $$

\end{proof}

\begin{lemma}
Define $$\zeta_{j,\mathbf{A}}(\mathbf{b}) := e^{2 \lambda n} |g_{\mathbf{a}_{j-1}' \mathbf{b}}'(x_{\mathbf{\mathbf{a}}_j})|$$
and $$ J_n := \{ e^{\varepsilon_0 n/2} \leq |\eta| \leq e^{2 \varepsilon_0 n}  \} .$$
There exists $\beta > 0$ such that, for $n$ large enough depending on $\varepsilon$,

$$  e^{- \lambda \delta K} \sum_{\mathbf{C} \in \mathcal{R}_{K+1}} e^{-\lambda \delta (2k+1) n}   \underset{\mathbf{C} \rightsquigarrow \mathbf{A}}{\underset{\mathbf{A} \in \mathcal{R}_{n+1}^{k+1}(\omega)}{\sum}} \iint_{U_{b(\mathbf{A})}^2 }  \Bigg{|} \underset{\mathbf{A} \leftrightarrow \mathbf{B}}{\sum_{\mathbf{B} \in \mathcal{R}_{n+1}^k}} e^{i \xi \left|\Delta_{\mathbf{A},\mathbf{B},\mathbf{C}}\right|(x,y) } \Bigg| d\nu(x) d\nu(y)  \quad \quad \quad \quad \quad \quad   $$
$$ \quad \quad \quad \quad \quad \quad \lesssim  e^{-\lambda \delta (2k+1) n} \sum_{\mathbf{A} \in \mathcal{R}_{n+1}^{k+1}(\omega)} \sup_{\eta \in J_n} \Bigg{|} \underset{\mathbf{A} \leftrightarrow \mathbf{B}}{\sum_{\mathbf{B} \in \mathcal{R}_{n+1}^k}} e^{i \eta \zeta_{1,\mathbf{A}}(\mathbf{b}_1) \dots \zeta_{k,\mathbf{A}}(\mathbf{b}_k) } \Bigg|  + e^{ \varepsilon \beta n} \left( e^{-  (\alpha \lambda - \varepsilon_0) n}  + e^{-  \varepsilon_0 n \delta_{\text{up}}/2} \right) .$$

\end{lemma}

\begin{proof}

Our goal is to carefully approximate $\Delta_{\mathbf{A},\mathbf{B},\mathbf{C}}$ by a product of derivatives of $g_{\mathbf{a}_{j-1}'\mathbf{b}_j}$, and then to renormalize the phase. Using arc length parameterization, our 1-dimensional setting allows us to apply the mean value theorem: for all $x,y \in U_{b(\mathbf{A})}$, there exists $z \in U_{b(\mathbf{A})}$ such that
$$ |\phi( f^K g_{\mathbf{C}} g_{\mathbf{A}* \mathbf{B}} (x) ) -  \phi( f^K g_{\mathbf{C}} g_{\mathbf{A}* \mathbf{B}} (y) )| = $$ $$ |\partial_u \phi(  f^K g_{\mathbf{C}'(\mathbf{A}* \mathbf{B})} z )| |\partial_u f^K(g_{\mathbf{C}'(\mathbf{A}* \mathbf{B})} z )| |g_{\mathbf{C}}'(g_{\mathbf{A} * \mathbf{B} } z)| \left[ \prod_{j=1}^{k} |g_{\mathbf{a}_{j-1}'\mathbf{b}_j}'(g_{\mathbf{a}_j' \mathbf{b}_{j+1}' \dots \mathbf{a}_{k-1}' \mathbf{b}_k}z) )| \right]  \ d^u(g_{\mathbf{a}_k} x,g_{\mathbf{a}_k} y) .$$

The estimates of section 3 gives

$$ \left| {\Delta_{\mathbf{A},\mathbf{B},\mathbf{C}}(x,y)} \right| \leq C_{\phi,\chi} e^{\varepsilon \beta n} e^{- (2k+1) \lambda n} .$$

We then relate $\Delta_{\mathbf{A},\mathbf{B}, \mathbf{C}}$ to the $\zeta_{j,\mathbf{A}}$. Set $$ \tilde{\Delta}_{\mathbf{A},\mathbf{B},\mathbf{C}}(x,y) := |\partial_u \phi(  f^K g_{\mathbf{C}'}(x_{\mathbf{a}_0} )| |\partial_u f^K(g_{\mathbf{C}'} (x_{\mathbf{a}_0}) )| |g_{\mathbf{C}'}(x_{\mathbf{a}_0})| \left[ \prod_{j=1}^{k} |g_{\mathbf{a}_{j-1}'\mathbf{b}_j}'(x_{\mathbf{a}_j}) )| \right]  \ d^u(g_{\mathbf{a}_k}x,g_{\mathbf{a}_0}y).$$
Then, as before, $$ \left| {\tilde{\Delta}_{\mathbf{A},\mathbf{B},\mathbf{C}}(x,y)} \right| \leq C_{\phi,\chi} e^{\varepsilon \beta n} e^{- (2k+1) \lambda n} .$$ 
Hence, using the fact that $|e^{s} - e^{t}| \leq e^{\max(s,t)} |s-t| $, we get

$$ \left| \Delta_{\mathbf{A},\mathbf{B},\mathbf{C}}(x,y) - \tilde{\Delta}_{\mathbf{A},\mathbf{B},\mathbf{C}}(x,y) \right| \leq C_{\phi,\chi} e^{\varepsilon \beta n} e^{-(2k+1)\lambda n} \left| \ln |\Delta_{\mathbf{A},\mathbf{B},\mathbf{C}}|(x,y) - \ln |\tilde{\Delta}_{\mathbf{A},\mathbf{B},\mathbf{C}}|(x,y) \right|.$$

Moreover, using the estimates of section 3, and exponentially vanishing variations of Hölder maps, we get:
\begin{itemize}
    \item $ \left| \ln|\partial_u \phi(  f^K g_{\mathbf{C}}(x_{\mathbf{a}_0} )| - \ln |\partial_u \phi(f^K g_\mathbf{C}(x_{\mathbf{a}_0}))| \right| \lesssim e^{\varepsilon \beta n }e^{- \alpha \lambda n} $
    \item $ \left| \ln |\partial_u f^K(g_{\mathbf{C}(\mathbf{A}* \mathbf{B})} z )| - \ln |\partial_u f^K(g_{\mathbf{C}'} x_{\mathbf{a}_0} )| \right|  \lesssim e^{\varepsilon \beta n }e^{-\alpha \lambda n} $
    \item $ \left| \ln |g_{\mathbf{C}}'(g_{\mathbf{A}* \mathbf{B}} z )| - \ln |g_{\mathbf{C}}' x_{\mathbf{a}_0} | \right|  \lesssim e^{\varepsilon \beta n }e^{- \alpha \lambda n} $
    \item $ \left| \ln|g_{\mathbf{a}_{j-1}'\mathbf{b}_j}'(g_{\mathbf{a}_j' \mathbf{b}_{j+1}' \dots \mathbf{a}_{k-1}' \mathbf{b}_k}z) )| - \ln|g_{\mathbf{a}_{j-1}' \mathbf{b_j}}'(x_{\mathbf{a}_j})| \right| \lesssim e^{\varepsilon \beta n} e^{-\alpha \lambda n} $
\end{itemize}

So that summing every estimates gives us $$ \left| \ln |\Delta_{\mathbf{A},\mathbf{B},\mathbf{C}}|(x,y) - \ln |\tilde{\Delta}_{\mathbf{A},\mathbf{B},\mathbf{C}}|(x,y) \right| \lesssim e^{\varepsilon \beta n} e^{- \alpha \lambda n}. $$

Hence, $$  \left| \Delta_{\mathbf{A},\mathbf{B},\mathbf{C}}(x,y) - \tilde{\Delta}_{\mathbf{A},\mathbf{B},\mathbf{C}}(x,y) \right| \lesssim e^{\varepsilon \beta n} e^{-(2k+1+\alpha) \lambda n} ,$$

which allows us to approximate our main integral as follows:

$$e^{-(2k+1) \delta \lambda n - \delta \lambda K} \Bigg{|}  \underset{\mathbf{C} \rightsquigarrow \mathbf{A}}{\underset{\mathbf{A} \in \mathcal{R}_{n+1}^{k+1}(\omega)}{\sum_{\mathbf{C} \in \mathcal{R}_{K+1}}}} \iint_{U_{b(\mathbf{A})}^2 }  \Big{|} \underset{\mathbf{A} \leftrightarrow \mathbf{B}}{\sum_{\mathbf{B} \in \mathcal{R}_{n+1}^k}} e^{i \xi \left|\Delta_{\mathbf{A},\mathbf{B},\mathbf{C}}\right| } \Big| d\nu \otimes d\nu -  \underset{\mathbf{C} \rightsquigarrow \mathbf{A}}{\underset{\mathbf{A} \in \mathcal{R}_{n+1}^{k+1}(\omega)}{\sum_{\mathbf{C} \in \mathcal{R}_{K+1}}}} \iint_{U_{b(\mathbf{A})}^2 }  \Big{|} \underset{\mathbf{A} \leftrightarrow \mathbf{B}}{\sum_{\mathbf{B} \in \mathcal{R}_{n+1}^k}} e^{i \xi \left|\tilde{\Delta}_{\mathbf{A},\mathbf{B},\mathbf{C}}\right| } \Big| d\nu \otimes d\nu \Bigg{|}  $$

$$ \lesssim |\xi| e^{-(2k+1+\alpha) \lambda n} \lesssim e^{-(\alpha \lambda - \varepsilon_0)n},$$

since $|\xi| \simeq  e^{(2k+1) \lambda n} e^{\varepsilon_0 n} $.

To conclude, we notice that $ |\xi| |\tilde{\Delta}_{\mathbf{A},\mathbf{B},\mathbf{C}}|$ can be written as a product like so:
$$ |\xi| |\tilde{\Delta}_{\mathbf{A},\mathbf{B},\mathbf{C}}|(x,y) =  \eta_{\mathbf{A},\mathbf{C}}(x,y)  \zeta_{1,\mathbf{A}}(\mathbf{b}_1) \dots  \zeta_{k,\mathbf{A}}(\mathbf{b}_k)  $$

where $$ \eta_{\mathbf{A},\mathbf{C}}(x,y) := |\xi| |\partial_u \phi(  f^K g_{\mathbf{C}'}(x_{\mathbf{a}_0} )| |\partial_u f^K(g_{\mathbf{C}'} (x_{\mathbf{a}_0}) )| |g_{\mathbf{C}'}(x_{\mathbf{a}_0})| e^{-2 k \lambda n }d^u(g_{\mathbf{a}_k}x,g_{\mathbf{a}_0}y). $$

We estimate $\eta_{\mathbf{A},\mathbf{C}}$ using the hypothesis made on $\partial_u \phi$, the estimates of section 3, and the mean value theorem, to get

$$  C_{\phi,\chi}^{-1} e^{-\varepsilon \beta n} e^{\varepsilon_0 n} d^u(x,y) \leq \eta_{\mathbf{A},\mathbf{C}}(x,y) \leq C_{\phi,\chi} e^{\varepsilon \beta n} e^{\varepsilon_0 n} \leq e^{2 \varepsilon_0 n} .$$

Notice that for the lower inequality to hold, it was critical for $\mathbf{a}_0$ to be in $\mathcal{R}_{n+1}(\omega)$, and not just in $\mathcal{R}_{n+1}$. \\

We then see that $\eta_{\mathbf{A},\mathbf{C}}(x,y) \in J_n$ as soon as $d^u(x,y) \geq C_{\phi,\chi} e^{\varepsilon \beta n - \varepsilon_0 n/2}$. To get rid of the part of the integral where $d^u(x,y)$ is too small, we use the upper regularity of $\nu$, proved in Lemma 2.9. For all $y \in \mathcal{U}$, the ball $B(y, C_{\phi,\chi} e^{\varepsilon \beta n - \varepsilon_0 n /2} )$ have measure $ \lesssim e^{-( \varepsilon_0/2  - \beta \varepsilon) \delta_{up} n} $, so that by integrating over $y$,

$$ \nu \otimes \nu \left( \{ (x,y) \in \mathcal{U}^2 \ | \ |x-y|< C_{\phi,\chi} e^{\varepsilon \beta n - \varepsilon_0 n/2} \} \right) \lesssim e^{-( \varepsilon_0/2  - \beta \varepsilon) \delta_{up} n} $$
as well. Hence we can cut the double integral in two, the part near the diagonal which is controlled by the previous estimates, and the part far away from the diagonal where $\eta_{\mathbf{A},\mathbf{C}}(x,y) \in J_n$.

Once this is done, the sum over $\mathbf{C}$ disappears, as there is no longer a dependence over $\mathbf{C}$ in the phase. \end{proof}

\section{The sum product phenomenon}

\subsection{A key theorem }

Our second move is to use a powerful theorem of Bourgain to control the sum of exponential. This version is the Proposition 3.2 of \cite{BD17}. This theorem was generalized by J. Li in \cite{Li18} and constitute the cornerstone of the method.

\begin{proposition}

Fix $\gamma > 0$. There exist $ \varepsilon_2 \in \ ]0,1[$ and $k \in \mathbb{N}^*$ such that the following holds for
$\eta \in \mathbb{R}$ with $|\eta| > 1$. Let $C_0 > 1$ and let $\lambda_1, \dots , \lambda_k$ be Borel measures supported on the interval $[C_0^{-1}, C_0] $ with total mass less than $C_0$. Assume that each $\lambda_j$ satisfies the following non concentration property:
$$\forall \sigma \in [ C_0 | \eta |^{-1} , C_0^{-1} |\eta|^{- \varepsilon_2}],  \quad  \lambda_j \otimes \lambda_j \left( \{ (x,y) \in \mathbb{R}^2, \ |x-y| \leq \sigma \} \right) \leq C_0 \sigma^\gamma .$$
Then there exists a constant $C_1$ depending only on $C_0$ and $\gamma$ such that
$$ \left| \int \exp( i \eta z_1 \dots z_k )  d\lambda_1(z_1) \dots d\lambda_k(z_k) \right| \leq C_1 |\eta|^{- \varepsilon_2} $$

\end{proposition}
Unfortunately, in our case the use of large deviations does not allow us to apply it straightforwardly. To highlight the dependence of $C_1$ when $C_0$ is permitted to grow gently, we prove the following proposition. The proof is the same than in the complex case done in $\cite{Le21}$, but we include it for completeness. Notice that the complex case does not immediately imply the real case, as the \emph{projective} non concentration hypothesis is stronger than our non concentration hypothesis.

\begin{proposition}

Fix $0< \gamma < 1$. There exist $ \varepsilon_1 > 0$ and $k \in \mathbb{N}$ such that the following holds for
$\eta \in \mathbb{R}$ with $|\eta|$ large enough. Let $1< R < |\eta|^{\varepsilon_1}$ and let $\lambda_1, \dots , \lambda_k$ be Borel measures supported on the interval $[R^{-1},R]$ with total mass less than $R$. Assume that each $\lambda_j$ satisfies the following non concentration property:
$$\forall \sigma \in [|\eta|^{-2}, |\eta|^{- \varepsilon_1}], \quad  \lambda_j \otimes \lambda_j \left( \{ (x,y) \in \mathbb{R}^2, \ |x-y| \leq \sigma \} \right) \leq \sigma^\gamma .$$
Then there exists a constant $c>0$ depending only on $\gamma$ such that
$$ \left| \int \exp(i \eta z_1 \dots z_k)  d\lambda_1(z_1) \dots d\lambda_k(z_k) \right| \leq c |\eta|^{- \varepsilon_1} $$

\end{proposition}

\begin{proof}

 Fix $0<\gamma<1$, and let $\varepsilon_2$ and $k$ given by the previous theorem. Choose $\varepsilon_1 := \frac{\varepsilon_2}{2(2k+1)}$. Let $1<R<|\eta|^{\varepsilon_1}$, and let $\lambda_1 , \dots, \lambda_k$ be measures that satisfy the hypothesis of Proposition 5.1. We are going to use a dyadic decomposition. \\

Let $m := \lfloor \log_2(R) \rfloor + 1$.
Then $\lambda_j$ is supported in the interval $[2^{-m}, 2^m ]$.
Define, for $A$ a borel subset of $\mathbb{R}$ and for $r=-m+1,\dots , m$:  $$ \lambda_{j,r}(A) := R^{-1} \lambda_j\left( 2^{r} \left( A \cap [1/2,1[ \right) \right) $$

Those measures are all supported in $ [1/2,1[$, and have total mass $\lambda_{j,r}(\mathbb{R}) \leq 1$. \\

Moreover, a non concentration property is satisfied by each $\lambda_{j,r}$. If we fix some $r_1,\dots,r_k$ between $-m+1$ and $m$ and define ${\eta}_{r_1\dots r_k} := 2^{r_1+\dots r_k} \eta$, then $|\eta_{r_1,\dots, r_k}| \geq (2R)^{-k} |\eta| > 2^{-k} |\eta|^{1-k \varepsilon_1} > 1$ if $\eta$ is large enough. Let $\sigma \in [|\eta_{r_1,\dots, r_k}|^{-1}, |\eta_{r_1,\dots, r_k}|^{-{\varepsilon_2}}]$. Then
$$\lambda_{j,r} \otimes \lambda_{j,r} \left( \{ (x,y) \in \mathbb{R}^2, \ |x-y| \leq \sigma \} \right) = \int_{\mathbb{R}} \lambda_{j,r}( [x-\sigma,x+\sigma] ) d\lambda_{j,r}(x) $$
$$ \leq R^{-2} \int_{\mathbb{R}} \lambda_j\left( [2^r x - 2^r\sigma , 2^r x + 2^r \sigma ] \right) d\lambda_j( 2^r x ) $$
$$ = R^{-2} \lambda_{j} \otimes \lambda_j \left( \{ (x,y) \in \mathbb{R}^2, \ |x-y| \leq 2^r \sigma \} \right) $$

Since $2^{r} \sigma \in \left[ 2^{r} |\eta_{r_1,\dots,r_k}|^{-1} , 2^{r} |\eta_{r_1,\dots,r_k}|^{- \varepsilon_2} \right] \subset \left[ (2R)^{-(k+1)} |\eta|^{-1} , (2R)^{k+1}|\eta|^{- \varepsilon_2}\right] \subset \left[|\eta|^{-2}, |\eta|^{\varepsilon_1}\right]$ if $|\eta|$ is large enough, we can use the non-concentration hypothesis assumed for each $\lambda_j$ to get:
$$ \lambda_{j,r} \otimes \lambda_{j,r} \left( \{ (x,y) \in \mathbb{R}^2, \ |x-y| \leq \sigma \} \right) \leq R^{-2} (2^{-r} \sigma)^\gamma \leq  \sigma^\gamma .$$

Hence, by the previous proposition, there exists a constant $C_1$ depending only on $\gamma$ such that
$$ \left| \int \exp( i \eta_{r_1 \dots r_k} z_1 \dots z_k)  d\lambda_{1,r_1}(z_1) \dots d\lambda_{k,r_k}(z_k) \right| \leq C_1 |\eta_{r_1 \dots r_k}|^{- \varepsilon_2} . $$
Finally, since $$ \lambda_j(A) = R \sum_{r=-m+1}^m \lambda_{j,r}(2^{-r} A) $$
we get that:
$$  \left| \int \exp( i  \eta z_1 \dots z_k)  d\lambda_1(z_1) \dots d\lambda_k(z_k) \right| $$
$$ \leq \sum_{r_1 , \dots r_k} R^k \left| \int \exp(i \eta z_1 \dots z_k)  d\lambda_{1,r_1}(2^{-r_1} z_1) \dots d\lambda_{k,r_k}(2^{-r_k} z_k) \right| $$
$$ =  \sum_{r_1 , \dots r_k} R^k \left| \int \exp(i \eta_{r_1 \dots r_k} z_1 \dots z_k)  d\lambda_{1,r_1}(z_1) \dots d\lambda_{k,r_k}(z_k) \right|$$
$$ \leq C_1 (2m)^k R^k |\eta_{r_1 \dots r_k}|^{- \varepsilon_2}  \leq 4^k C_1 m^k R^{2k} |\eta|^{- \varepsilon_2} $$

Since $m \leq \log_2(R) +1$, and since $k$ depends only on $\gamma$, there exists a constant $c$ that depends only on $\gamma$ such that $ 4^k C_1 m^k R^{2k} \leq c R^{2k+1}$ for any $R> 1$. Finally, $ c R^{2k+1} |\eta|^{- \varepsilon_2} \leq |\eta|^{-\varepsilon_1} $. \end{proof}

\begin{corollary}

Fix $0 < \gamma < 1$. There exist $k \in \mathbb{N}^*$ and $\varepsilon_1 > 0$ depending only on $\gamma$ such that the following holds for $\eta \in \mathbb{R}$ with $|\eta|$ large enough. Let $1 < R < |\eta|^{\varepsilon_1}$ , $N > 1$ and $\mathcal{Z}_1,\dots , \mathcal{Z}_k$ be finite sets such that $ \# \mathcal{Z}_j \leq RN$. Consider some maps $\zeta_j : \mathcal{Z}_j \rightarrow \mathbb{R} $, $j = 1, \dots , k$, such that, for all $j$:
$$ \zeta_j ( \mathcal{Z}_j ) \subset [R^{-1},R] $$ and 
$$\forall \sigma \in [|\eta|^{-2}, |\eta|^{- \varepsilon_1}], \quad \# \{\mathbf{b} , \mathbf{c} \in \mathcal{Z}^2_j , \ |\zeta_j(\mathbf{b}) - \zeta_j(\mathbf{c})| \leq \sigma \} \leq N^2 \sigma^{\gamma}.$$
Then there exists a constant $c > 0$ depending only on $\gamma$ such that 
$$ \left| N^{-k} \sum_{\mathbf{b}_1 \in \mathcal{Z}_1,\dots,\mathbf{b}_k \in \mathcal{Z}_k} \exp\left( i \eta \zeta_1(\mathbf{b}_1) \dots \zeta_k(\mathbf{b}_k) \right) \right| \leq c |\eta|^{-{\varepsilon_1}}$$

\end{corollary}

\begin{proof}

Define our measures as sums of dirac mass: $$ \lambda_j := \frac{1}{N} \sum_{\mathbf{b} \in \mathcal{Z}_j} \delta_{\zeta_j(\mathbf{b})} .$$

We see that $\lambda_j$ is supported in  $ [R^{-1},R] $. The total mass is bounded by
$$ \lambda_j(\mathbb{R}) \leq N^{-1} \# \mathcal{Z}_j \leq R.$$

Then, if $\sigma \in [|\eta|^{-2}, |\eta|^{- \varepsilon_1}]$, we have, for any $a \in \mathbb{R}$:
$$ \lambda_j \otimes \lambda_j \left( \{ (x,y) \in \mathbb{R}^2 , \ |x-y|< \sigma \} \right) = \frac{1}{N^2} \# \left\{\mathbf{b} , \mathbf{c} \in \mathcal{Z}^2_j  , \  |\zeta_j(\mathbf{b}) - \zeta_j(\mathbf{c}) | \leq \sigma \right\} \leq  \sigma^\gamma. $$
Hence, the previous theorem applies directly, and gives us the desired result. \end{proof}

\subsection{End of the proof assuming non concentration}

We will use Corollary 5.3 on the maps $\zeta_{j,\mathbf{A}}$.
Let's carefully define the framework. \\
For some fixed $\mathbf{A} \in \mathcal{R}_{n+1}^{k+1}(\omega)$, define for $j=1, \dots, k$  $$ \mathcal{Z}_j := \{ \mathbf{b} \in \mathcal{R}_{n+1} , \mathbf{a}_{j-1} \rightsquigarrow \mathbf{b} \rightsquigarrow \mathbf{a}_j \  \} ,$$

so that the maps $\zeta_{j,\mathbf{A}}(\mathbf{b}) := e^{2 \lambda n} |g_{\mathbf{a}_{j-1}' \mathbf{b}}'(x_{{\mathbf{a}}_j})|$ are defined on $\mathcal{Z}_j$. There exists a constant $\beta>0$  such that:
$$\# \mathcal{Z}_j \leq e^{\varepsilon \beta n} e^{ \delta \lambda n} $$
and
$$ \zeta_{j,\mathbf{A}}( \mathcal{Z}_j ) \subset \left[ e^{- \varepsilon \beta n}, e^{\varepsilon \beta n} \right] .$$

Let $\gamma>0$ be small enough. The theorem 5.3 then fixes $k$ and some $\varepsilon_1$.
The goal is to apply Corollary 5.3 to the maps $\zeta_{j,\mathbf{A}}$, for $N := e^{\lambda \delta n}$, $R:=e^{\varepsilon \beta n}$ and $\eta \in J_n$. Notice that choosing $\varepsilon$ small enough ensures that $R<|\eta|^{\varepsilon_1}$, and taking $n$ large enough ensures that $|\eta|$ is large. 
If we are able to prove the non concentration hypothesis in this context, then Corollary 5.3 can be applied and we would be able to conclude the proof of the main Theorem 1.2.
Indeed, we already know that
$$ |\widehat{\mu}(\xi)|^2 \lesssim e^{ \varepsilon \beta n} e^{-\lambda \delta (2k+1) n} \sum_{\mathbf{A} \in \mathcal{R}_{n+1}^{k+1}(\omega)} \sup_{\eta \in J_n} \Bigg{|} \underset{\mathbf{A} \leftrightarrow \mathbf{B}}{\sum_{\mathbf{B} \in \mathcal{R}_{n+1}^k}} e^{i \eta \zeta_{1,\mathbf{A}}(\mathbf{b}_1) \dots \zeta_{k,\mathbf{A}}(\mathbf{b}_k) } \Bigg| $$
$$ \quad \quad  \quad \quad  \quad \quad  \quad \quad + e^{- \varepsilon_0 n} +  e^{-\delta_1(\varepsilon) n} + e^{\varepsilon \beta n} \left( e^{- \lambda \alpha n} +  \kappa^{\alpha n} + e^{-(\alpha \lambda-\varepsilon_0)n} +  e^{- \varepsilon_0 \delta_{up}n/2} \right) $$

by Proposition 4.1. Since every error term already enjoys exponential decay in $n$, we just have to deal with the sum of exponentials. By Corollary 5.3, we can then write
$$ \sup_{\eta \in J_n} \Bigg{|} \underset{\mathbf{A} \leftrightarrow \mathbf{B}}{\sum_{\mathbf{B} \in \mathcal{R}_{n+1}^k}} e^{i \eta \zeta_{1,\mathbf{A}}(\mathbf{b}_1) \dots \zeta_{k,\mathbf{A}}(\mathbf{b}_k) } \Bigg| \leq c e^{\lambda k \delta n} e^{- \varepsilon_0 \varepsilon_1 n/2 } ,$$

and hence we get
$$ e^{-\lambda \delta (2k+1) n} \sum_{\mathbf{A} \in \mathcal{R}_{n+1}^{k+1}(\omega)} \sup_{\eta \in J_n} \Bigg{|} \underset{\mathbf{A} \leftrightarrow \mathbf{B}}{\sum_{\mathbf{B} \in \mathcal{R}_{n+1}^k}} e^{ i  \eta \zeta_{1,\mathbf{A}}(\mathbf{b}_1) \dots \zeta_{k,\mathbf{A}}(\mathbf{b}_k)} \Bigg| $$ $$ \lesssim e^{\varepsilon \beta n} e^{-\lambda \delta  (2k+1) n} e^{\lambda \delta (k+1) n} e^{\lambda \delta k n} e^{- \varepsilon_0 \varepsilon_1 n/2} \lesssim e^{\varepsilon \beta n} e^{-\varepsilon_0 \varepsilon_1 n/2} .$$

Now, we see that we can choose $\varepsilon$ small enough so that all terms enjoy exponential decay in $n$, and since $ |\xi| \simeq e^{\left( (2k+1) \lambda + \varepsilon_0 \right) n} $, we have proved polynomial decay of $|\widehat{\mu}|^2$.

\section{The non-concentration hypothesis}

The last part of this paper is devoted to the proof of the non-concentration hypothesis that we just used. The strategy is the same than in $\cite{SS20}$ and $\cite{Le21}$, but the theorem used to conclude will be more recent.

\begin{definition}

For a given $\mathbf{A} \in \mathcal{R}_{n+1}^{k+1}(\omega)$, define for $j=1, \dots, k$  $$ \mathcal{Z}_j := \{ \mathbf{b} \in \mathcal{R}_{n+1} , \ \mathbf{a}_{j-1} \rightsquigarrow \mathbf{b} \rightsquigarrow \mathbf{a}_j \  \} $$

Then define $$\zeta_{j,\mathbf{A}}(\mathbf{b}) := e^{2 \lambda n} |g_{\mathbf{a}_{j-1}' \mathbf{b}}'(x_{{\mathbf{a}}_j})|$$ on $\mathcal{Z}_j$. The following is satisfied, for some fixed constant $\beta>0$:

 $$\# \mathcal{Z}_j \leq e^{\varepsilon \beta n} e^{ \delta \lambda n} $$
 and
$$ \zeta_{j,\mathbf{A}}( \mathcal{Z}_j ) \subset \left[ e^{- \varepsilon \beta n} , e^{\varepsilon \beta n} \right] .$$

\end{definition}

We are going to prove the following fact, which will allow us to apply Corollary for $\eta \in J_n$, $R:=e^{\varepsilon \beta n}$ and $N:= e^{\lambda \delta n} $:

\begin{proposition}[non concentration]
There exists $\gamma>0$, and we can choose $\varepsilon_0>0$, such that the following holds. Let $\eta \in J_n$. Let $\mathbf{A} \in \mathcal{R}_{n+1}^{k+1}$. Then, if $n$ is large enough,

$$\forall \sigma \in [ |\eta|^{-2}, |\eta|^{- \varepsilon_1}], \quad \sup_{a \in \mathbb{R}} \ \# \left\{(\mathbf{b}, \mathbf{c}) \in \mathcal{Z}^2_j , \  |\zeta_{j,\mathbf{A}}(\mathbf{b})-\zeta_{j,\mathbf{A}}(\mathbf{c}) | \leq \sigma \right\} \leq N^2 \sigma^{\gamma}. $$

where $R:= e^{\varepsilon \beta n}$, $N:= e^{\lambda \delta n}$ and $\varepsilon_1,k$ are fixed by Corollary 5.3.

\end{proposition}

The proof of Proposition 6.1 is based on a uniform spectral gap for a family of twisted transfer operators, as in \cite{SS20}, also known as Dolgopyat's estimates. Before introducing it, let us first reduce our non-concentration estimate to a statement about Birkhoff sums.

\begin{lemma}

If $\varepsilon_0$ and $\gamma$ are such that, for $\sigma \in [ e^{- 5 \varepsilon_0 n} , e^{ - \varepsilon_1 \varepsilon_0 n/4 }  ]$, $$\sup_{a \in \mathbb{R}} \#  \left\{ \mathbf{b} \in \mathcal{Z}_j,  \ S_{2n} {\tau}_F \left( g_{\mathbf{a}_{j-1}' \mathbf{b}} (x_{\mathbf{a}_j})  \right) \in [a-\sigma,a+\sigma] \right\} \leq N \sigma^{2 \gamma} ,$$

Then Proposition 6.1 is true.

\end{lemma}

\begin{proof}

Suppose that the estimate is true.
Let $|\eta| \in [e^{\varepsilon_0 n/2}, e^{2 \varepsilon_0 n}] $, and then let $\sigma \in [ |\eta|^{-2}, |\eta|^{-\varepsilon_1} ] \subset [e^{-4 \varepsilon_0 n}, e^{- \varepsilon_0 \varepsilon_1 n/2}].$  Let $a \in [R^{-1},R]$ (it is enough to conclude). Since for $n$ large enough $$\ln(a+\sigma)-\ln(a-\sigma) = \ln(1+ \sigma a^{-1}) - \ln(1 - \sigma a^{-1}) \leq 4 \sigma a^{-1} \leq 4 \sigma R ,$$ 

We find that
$$ \ln\left( [a-\sigma,a+\sigma] \right) \subset [\ln a - 4 R \sigma, \ \ln a + 4 R \sigma ] . $$

Hence:
$$  \#  \{ \mathbf{b} \in \mathcal{Z}_j,  \ {\zeta}_{j,\mathbf{A}}(\mathbf{b}) \in [a-\sigma,a+\sigma] \} \leq  \#  \{ \mathbf{b} \in \mathcal{Z}_j,  \ \ln {\zeta}_{j,\mathbf{A}}(\mathbf{b}) \in [\ln a - 4 R \sigma, \ln a + 4 R \sigma]  \} $$
$$ = \#  \{ \mathbf{b} \in \mathcal{Z}_j,  \ S_{2n} {\tau}_F \left( g_{\mathbf{a}_{j-1}' \mathbf{b}} (x_{\mathbf{a}_j})  \right)  \in \left[ -\ln a + 2 n \lambda - 4 \sigma R , -\ln a + 2 n \lambda + 4 \sigma R \right] \} $$
$$ \leq N (4 R \sigma)^{2 \gamma} \leq N \sigma^{\gamma} $$
 since $4R\sigma \in [e^{-5 \varepsilon_0 n}, e^{-\varepsilon_1 \varepsilon_0 n/4}]$ for large enough $n$. Finally,

$$ \# \left\{(\mathbf{b}, \mathbf{c}) \in \mathcal{Z}^2_j , \  |{\zeta}_{j,\mathbf{A}}(\mathbf{b})-{\zeta}_{j,\mathbf{A}}(\mathbf{c}) | \leq \sigma \right\}  \quad \quad \quad \quad \quad \quad \quad \quad \quad \quad \quad $$ $$ \quad \quad \quad \quad \quad = \sum_{\mathbf{c} \in \mathcal{Z}_j}   \left\{ \mathbf{b} \in \mathcal{Z}_j,  \ {\zeta}_{j,\mathbf{A}}(\mathbf{b}) \in [\zeta_{j,\mathbf{A}}(\mathbf{c})-\sigma,\zeta_{j,\mathbf{A}}(\mathbf{c})+\sigma] \right\} \leq N^2 \sigma^{\gamma}, $$

and so proposition 6.1 is true. \end{proof}

To prove that the estimate of Lemma 6.2 is satisfied, we use the following spectral gap for twisted transfer operators, established in \cite{SS20} (Th. 5.1) for full branched expanding maps. A similar statement can be found in \cite{DV21} (Th. 6.4) in a more general setting. This kind of theorem is not new and comes from the early work of Dolgopyat \cite{Do98}. This recent version is more adapted in our context, as it deals explicitly with one dimensional shifts that satisfies a nonlinearity condition that is easily seen to be true in our context.  \\

Unfortunately, the use of Dolgopyat's estimates can only be checked if the function $\tau_F$ is smooth enough (Hölder regular is \textbf{not} enough). This is were we need our additional bunching assumption (B) to conclude: if our attractor contract strongly enough in the stable direction, then the stable lamination $W^s_\varepsilon(x)$ becomes $C^{2+\alpha}$. In particular, the map $\pi$ is $C^{2+\alpha}$ in the unstable direction, and so $F$ is $C^{2+\alpha}$ too. In this case, $\tau_F$ is $C^{1+\alpha}$, which is enough to use Dolgopyat's estimates.

\begin{theorem}
Define, for $s \in \mathbb{C}$, a twisted transfer operator $\mathcal{L}_{s} : C^\alpha(\mathcal{U},\mathbb{C}) \rightarrow C^\alpha(\mathcal{U},\mathbb{C}) $ as follows:

$$ \forall x \in U_b, \ \mathcal{L}_{s} h (x) := \sum_{a \rightarrow b} e^{(\varphi+s \tau_F)(g_{ab}(x))} h(g_{ab}(x)) $$

Iterating this transfer operator yields:

$$ \forall x \in U_b, \ \mathcal{L}_{s}^n h (x) = \underset{\mathbf{a} \rightsquigarrow b}{\sum_{\mathbf{a} \in \mathcal{W}_{n+1}}} w_{\mathbf{a}}(x) e^{s S_{n} {\tau}_F(g_{\mathbf{a}}(x)) } h(g_{\mathbf{a}}(x)) .$$

Under our nonlinearity condition (NL) and the bunching condition (B), the following holds. There exists $\rho \in \mathbb{N}$ such that, for any $s \in \mathbb{C}$ such that $\text{Re}(s)=0$ and $|\text{Im}(s)|>\rho$,

$$ \forall h \in C^\alpha(\mathcal{U},\mathbb{C}), \ \forall n \geq 0, \ \| \mathcal{L}_{s}^n h \|_{\infty, \mathcal{U}} \leq \rho |\text{Im}(s)|^{\rho} e^{-n/\rho} \|h\|_{C^\alpha(\mathcal{U},\mathbb{C})} .$$

\end{theorem}

It means that this twisted transfer operator is eventually contracting for large $\text{Im}(s)$. This theorem will play another key role in this paper. 

\begin{remark}

This theorem is stated in $\cite{DV21}$ under a "total non linearity hypothesis" made on $\tau_F$. The condition goes as follow: there exists no locally constant map  $\mathfrak{c} : \mathcal{U}^{(1)} \rightarrow \mathbb{R}$  and  $\theta \in C^1(\mathcal{U}^{(1)},\mathbb{R})$  such that  $$\tau_F = \mathfrak{c} - \theta \circ F + \theta. $$

This condition is satisfied in our setting. Indeed, suppose that some locally constant $\mathfrak{c} : \mathcal{U}^{(1)} \rightarrow \mathbb{R}$ satisfy $ \tau_F = \mathfrak{c} + \theta \circ F - \theta $ for some $\mathcal{C}^1$ map $\theta$. Then, recall that $\tau_F \circ \pi$ and $\tau_f$ are $f$-cohomologous (Lemma 2.13). Hence, if $x \in \Omega_{per}$ is a $f$-periodic point with period $n_x$, the unstable Lyapunov exponent of the associated periodic orbit is $$ \widehat{\lambda}(x) = \frac{1}{n_x} S_{n_x} \tau_F(\pi(x)) \in \text{Span}_\mathbb{Q} \left( \mathfrak{c}(\mathcal{U})\right) , $$
which implies that
$$ \text{dim}_\mathbb{Q} \text{Vect}_\mathbb{Q} \widehat{\lambda}\left( \Omega_\text{per} \right) \leq  \text{dim}_\mathbb{Q}  \text{Span}_\mathbb{Q} \left( \mathfrak{c}(\mathcal{U})\right) , $$

and this is an obvious contradiction to the nonlinearity hypothesis (NL). \\

\end{remark}

\begin{lemma}

Define $ \varepsilon_0 := \min\left( 1/(5 \rho (3+\rho)), \alpha \lambda/8  \right)$. Fix $\gamma:=1/4$, and let $\varepsilon_1$ and $k$ be fixed by Theorem 5.3. For $\sigma \in [e^{-5 \varepsilon_0 n} , e^{ - \varepsilon_1 \varepsilon_0 n/4 }  ]$ and if $n$ is large enough, $$ \sup_{a \in \mathbb{R}} \#  \left\{ \mathbf{b} \in \mathcal{Z}_j,  \ S_{2n} {\tau}_F \left( g_{\mathbf{a}_{j-1}' \mathbf{b}} (x_{\mathbf{a}_j})  \right) \in [a-\sigma,a+\sigma] \right\} \leq N \sigma^{1/2} .$$

\end{lemma}

\begin{proof}
In the proof to come, all the $\simeq$ or $\lesssim$ will be uniform in $a$: the only relevant information is $\sigma$.
So fix $\sigma \in [ e^{- 5 \varepsilon_0 n} , e^{ - \varepsilon_1 \varepsilon_0 n/4 }  ]$, and fix an interval of length $\sigma$, $[a-\sigma,a+\sigma]$. Then choose a bump function $\chi$ such that $\chi = 1$ on $[a-\sigma,a+\sigma]$, $\text{supp}(\chi) \subset [a-2\sigma,a+2\sigma]$ and such that $ \| \chi \|_{L^1(\mathbb{R})} \simeq \sigma $. We can suppose that $ \| \chi^{(l)} \|_{L^1(\mathbb{R})} \simeq \sigma^{1-l} $.  \\

Then, we can consider $h$, the $2 \pi \mathbb{Z}$ periodic map obtained by periodizing $\chi$. This will allows us to use Fourier series. By construction, we see that
$ \mathbb{1}_{ [a-\sigma, a+\sigma] } \leq h .$
Hence:

$$ \#  \left\{ \mathbf{b} \in \mathcal{Z}_j,  \ S_{2n} {\tau}_F ( g_{\mathbf{a}_{j-1}' \mathbf{b}} (x_{\mathbf{a}_j})  ) \in [a-\sigma,a+\sigma] \right\} \leq \sum_{\mathbf{b} \in \mathcal{Z}_j} h\left( S_{2n} {\tau}_F ( g_{\mathbf{a}_{j-1}' \mathbf{b}} (x_{\mathbf{a}_j})  ) \right) $$
$$ \leq R N{\sum_{\mathbf{b} \in \mathcal{Z}_{j}}} w_{ \mathbf{b}}(x_{\mathbf{a}_j}) h\left( S_{2n} {\tau}_F ( g_{\mathbf{a}_{j-1}' \mathbf{b}} (x_{\mathbf{a}_j}) ) \right) $$
$$ \leq R N \underset{\mathbf{a}_{j-1} \rightsquigarrow \mathbf{b} \rightsquigarrow \mathbf{a}_j}{\sum_{\mathbf{b} \in \mathcal{W}_{n+1}}} w_{ \mathbf{b}}(x_{\mathbf{a}_j})  h\left( S_{2n} \tau_F ( g_{\mathbf{a}_{j-1}' \mathbf{b}} (x_{\mathbf{a}_j}) ) \right) .$$
Then we develop $h$ using Fourier series. We have:
$$ \forall x \in \mathbb{R}, \ h(x) = \sum_{ \mathfrak{n} \in \mathbb{Z} } c_{ \mathfrak{n} }(h) e^{i  \mathfrak{n} x} ,$$

where $$c_{\mathfrak{n}}(h) := (2 \pi)^{-1} \int_{a-\pi}^{a+\pi} h(x) e^{-i  \mathfrak{n} x} dx .$$ 
Notice that $$  \mathfrak{n}^{l} |c_{ \mathfrak{n}}(h)| \simeq |c_{ \mathfrak{n}}(h^{(l)})| \lesssim  \sigma^{ 1 - l} .$$

Plugging $ S_{2n} \tau_F ( g_{\mathbf{a}_{j-1}' \mathbf{b}} (x_{\mathbf{a}_j}) ) $ in this expression yields

$$ h\left( S_{2n} \tau_F ( g_{\mathbf{a}_{j-1}' \mathbf{b}} (x_{\mathbf{a}_j}) ) \right) = \sum_{ \mathfrak{n} \in \mathbb{Z} } c_{\mathfrak{n}}(h) e^{i \mathfrak{n} S_{2n} \tau_F ( g_{\mathbf{a}_{j-1}' \mathbf{b}} (x_{\mathbf{a}_j}) )} ,$$

and so
$$  \#  \left\{ \mathbf{b} \in \mathcal{Z}_j,  \ S_{2n} {\tau}_F ( g_{\mathbf{a}_{j-1}' \mathbf{b}} (x_{\mathbf{a}_j}) ) \in [a-\sigma,a+\sigma] \right\} $$
$$ \leq R N  \sum_{\mathfrak{n} \in \mathbb{Z}} c_{\mathfrak{n} }(h) \underset{\mathbf{a}_{j-1} \rightsquigarrow \mathbf{b} \rightsquigarrow \mathbf{a}_j}{\sum_{\mathbf{b} \in \mathcal{W}_{n+1}}} w_{ \mathbf{b}}(x_{\mathbf{a}_j}) e^{i \mathfrak{n} S_{2n} \tau_F ( g_{\mathbf{a}_{j-1}' \mathbf{b}} (x_{\mathbf{a}_j}) )} .$$

For any block $\mathbf{A} \in \mathcal{W}_{n+1}$, integer $j \in \{1, \dots, k\}$ and integer $\mathfrak{n} \in \mathbb{Z}$, define $\mathfrak h_{\mathbf{A},j}^{\mathfrak{n}} \in   C^\alpha(\mathcal{U},\mathbb{R})$ by
$$ \forall x \in U_{b(\mathbf{a})}, \ \mathfrak h_{\mathbf{A},j}^{\mathfrak{n}} (x) := \exp\left({ i \mathfrak{n} S_n \tau_F(g_{\mathbf{a}_{j-1}}(x))} \right) \quad , \quad \forall x \in U_b , \ b \neq b(\mathbf{a}), \  \mathfrak h_{\mathbf{A},j}^{\mathfrak{n}} (x) := 0 .$$
With this notation, we may rewrite the sum on $\mathbf{b}$ as follows:
$$  \underset{\mathbf{a}_{j-1} \rightsquigarrow \mathbf{b} \rightsquigarrow \mathbf{a}_j}{\sum_{\mathbf{b} \in \mathcal{W}_{n+1}}} w_{ \mathbf{b}}(x_{\mathbf{a}_j}) \exp\left({i \mathfrak{n} S_{2n} {\tau}_F ( g_{\mathbf{a}_{j-1}' \mathbf{b}} (x_{\mathbf{a}_j}) )} \right)$$
$$ = \underset{\mathbf{a}_{j-1} \rightsquigarrow \mathbf{b} \rightsquigarrow \mathbf{a}_j}{\sum_{\mathbf{b} \in \mathcal{W}_{n+1}}}    w_{\mathbf{b}}(x_{\mathbf{a}_j}) e^{i \mathfrak{n} S_n {\tau}_F( g_{\mathbf{b}}(x_{\mathbf{a}_j}) ) } \exp\left({ i \mathfrak{n} S_n {\tau}_F(g_{\mathbf{a}_{j-1}} g_{\mathbf{b}}(x_{\mathbf{a}_j}))} \right) $$
$$ = \underset{\mathbf{b} \rightsquigarrow \mathbf{a}_{j} }{\sum_{\mathbf{b} \in \mathcal{W}_{n+1}}} w_{\mathbf{b}}(x_{\mathbf{a}_j}) e^{i \mathfrak{n} S_n {\tau}_F( g_{\mathbf{b}}(x_{\mathbf{a}_j}) ) } \mathfrak h_{\mathbf{A},j}^{\mathfrak{n}} \left( g_{\mathbf{b}}(x_{\mathbf{a}_j}) \right) $$
$$ = \mathcal{L}_{i \mathfrak{n}}^n \left( \mathfrak{h}_{\mathbf{A},j}^{\mathfrak{n}} \right)(x_{\mathbf{a}_j}). $$

A direct computation allows us to estimate the $C^\alpha$ norm of $\mathfrak{h}_{\mathbf{A},j}^{\mathfrak{n}}$. We get, uniformly in $n$:
$$ \| \mathfrak{h}_{\mathbf{A},j}^{\mathfrak{n}} \|_{C^{\alpha}(\mathcal{U},\mathbb{R})} \lesssim (1+\mathfrak{n}) .$$

We can now break the estimate into two pieces: high frequencies are controlled by the contraction property of this transfer operator, and the low frequencies are controlled by the Gibbs property of $\mu$. We also use the estimates on the Fourier coefficients on $h$. 

$$ \#  \left\{ \mathbf{b} \in \mathcal{Z}_j,  \ S_{2n} {\tau}_F ( g_{\mathbf{a}_{j-1}' \mathbf{b}} (x_{\mathbf{a}_j}) ) \in [a-\sigma,a+\sigma] \right\}   $$
$$\leq R N  \sum_{\mathfrak{n} \in \mathbb{Z}} c_{\mathfrak{n} }(h) \underset{\mathbf{a}_{j-1} \rightsquigarrow \mathbf{b} \rightsquigarrow \mathbf{a}_j}{\sum_{\mathbf{b} \in \mathcal{W}_{n+1}}} w_{ \mathbf{b}}(x_{\mathbf{a}_j}) e^{i \mathfrak{n} S_{2n} \tau_F ( g_{\mathbf{a}_{j-1}' \mathbf{b}} (x_{\mathbf{a}_j}) )}  $$
$$ \leq R N \left( \sum_{|\mathfrak{n}| \leq \rho} |c_{\mathfrak{n}}(h)|  \sum_{\mathbf{b} \in \mathcal{W}_{n+1}}  w_{\mathbf{b}}(x_{\mathbf{a}_j})  +  \sum_{|\mathfrak{n}| > \rho} |c_{\mathfrak{n}}(h)| |\mathcal{L}_{i \mathfrak{n}}^n (\mathfrak{h}_{\mathbf{A},j})(x_{\mathbf{a}_j})| \right) $$
$$ \lesssim R N \left(  \sigma \sum_{\mathbf{b} \in \mathcal{W}_{n+1}} \mu(P_{\mathbf{b}})  + \sum_{|\mathfrak{n}| > \rho} |c_{\mathfrak{n}}(h)| \|\mathcal{L}_{ i \mathfrak{n}}^n (\mathfrak{h}_{\mathbf{A},j})\|_{\infty,\mathcal{U}} \right) $$
$$ \lesssim  R N \sigma +R N  \sum_{|\mathfrak{n} | > \rho} |c_{\mathfrak{n}}(h)| \mathfrak{n}^{\rho}  e^{-n/\rho} \| \mathfrak{h}_{\mathbf{A},j}^{\mathfrak{n}} \|_{C^\alpha(\mathfrak{U},\mathbb{C})} $$

$$ \lesssim  R N \sigma + R N e^{-n/\rho} \sum_{|\mathfrak{n}| > \rho} |c_{\mathfrak{n}}(h)| |\mathfrak{n}|^{\rho+3} |\mathfrak{n}|^{-2}    $$
$$ \leq C R N (\sigma + e^{-n/\rho} \sigma^{-(\rho+2)}) , $$
for some constant $C>0$.
We are nearly done. Since $\sigma \in  [ e^{-5 \varepsilon_0 n} , e^{ - \varepsilon_1 \varepsilon_0 n/4 }  ]$, we know that $ \sigma^{-(\rho+2)} \leq e^{5 (\rho+2) \varepsilon_0 n } $.
Now is the time where we fix $\varepsilon_0$: choose $$ \varepsilon_0 := \min\left( \frac{1}{5 \rho (3+\rho)}, \frac{\alpha \lambda}{8}  \right). $$

Then $ e^{-n/\rho} \sigma^{-(\rho+2)} \leq  e^{ ( -1/\rho + 5(\rho+2) \varepsilon_0 ) n } \leq  e^{ - 5 \varepsilon_0 n} \leq \sigma $ for $n$ large enough. Hence, we get

$$  \#  \left\{ \mathbf{b} \in \mathcal{Z}_j,  \ S_{2n} \tau_F ( g_{\mathbf{a}_{j-1}' \mathbf{b}} (x_{\mathbf{a}_j}) ) \in [a-\sigma,a+\sigma] \right\}  \leq 2 C R N \sigma . $$

Finally, since $\sigma^{1/2}$ is quickly decaying compared to $R$, we have

$$ 2 C R N \sigma \leq N \sigma^{1/2} $$

provided $n$ is large enough. The proof is done.
\end{proof}

\begin{appendices}

\section{The nonlinearity condition is generic}

\begin{theorem}

The condition $\text{dim}_\mathbb{Q} \text{Vect}_\mathbb{Q} \widehat{\lambda}\left( \Omega_\text{per} \right) = \infty$ is generic in the following sense: for any given Axiom A diffeomorphism $f : M \rightarrow M$ and a fixed basic set $\Omega$ that is not an isolated periodic orbit, a generic $C^1$ perturbation $\tilde{f}$ of $f$ have an hyperbolic set $\tilde{\Omega}$ on which the dynamic is conjugated with  $(f,\Omega)$, and $\text{dim}_\mathbb{Q} \text{Span}_\mathbb{Q} \widehat{\lambda}\left( \tilde{\Omega}_\text{per} \right) = \infty$.

\end{theorem}

\begin{proof}

Let $f:M \longrightarrow M$ be a $C^2$ Axiom A diffeomorphism, and fix $\Omega$ a basic set for $f$ that is not an isolated periodic orbit. Recall that this implies that $\Omega$ is infinite (and even perfect). Since $f$ is Axiom A, $\Omega_{\text{per}}$ is then infinite. Let $U \supset \Omega$ be a small open neighborhood in $M$. Consider a small enough open neighborhood around $f$ in the space of $C^1$ maps, $\mathfrak{U} \subset C^1(M,M)$. Then, there exists a map $$\Phi: \mathfrak{U} \rightarrow 2^M \times \mathcal{C}^0(\Omega,M) $$ 
such that for any $\tilde{f} \in \mathfrak{U}$, $\Phi(\tilde{f}) = (\tilde{\Omega},h)$ satisfies the following properties:

\begin{itemize}
     \item $\tilde{\Omega} \subset U$ is a hyperbolic set for $\tilde{f} : M \rightarrow M$,
    \item $\tilde{\Omega}_{\text{per}}$ is dense in $\tilde{\Omega}$,
    \item $h: \Omega \rightarrow \tilde{\Omega}$ is an homeomorphism and conjugates $(\Omega,f)$ with $(\tilde{\Omega},\tilde{f})$.
    \item The map $\tilde{f} \in \mathfrak{U} \mapsto h \in C^0(\Omega,M)$ is continuous.
\end{itemize}

This is Theorem 5.5.3 in \cite{BS02}. See also \cite{KH95}, page 571: \say{stability of hyperbolic sets}. \\
So let $\tilde{f} \in \mathfrak{U}$. Since $h$ conjugates $(f,\Omega)$ with $(\tilde{f},\tilde{\Omega})$, it is also a bijection between $\Omega_\text{per}$ and $\tilde{\Omega}_\text{per}$. For $x \in \Omega_\text{per}$, we write $\tilde{x} \in \tilde{\Omega}$ for $h(x)$. \\

We are ready to prove cleanly that the condition $$ \dim_\mathbb{Q} \text{Vect}_\mathbb{Q} \widehat{\lambda}\left( \tilde{\Omega}_{\text{per}} \right) = \infty$$
is generic in $\tilde{f} \in \mathfrak{U}$. The condition may be rewritten in the following way:

$$ \left\{ \tilde{f} \in \mathfrak{U} \ | \ \forall N \geq 1, \exists \tilde{x_1}, \dots, \tilde{x_N} \in \tilde{\Omega}_{\text{per}} \ , \ \left(\widehat{\lambda}(\tilde{x}_1) , \dots, \widehat{\lambda}(\tilde{x}_N) \right) \ \text{is linearly independant over} \ \mathbb{Q} \right\}. $$

ie
$$ \bigcap_{N \geq 1} \bigcup_{ \Omega_{\text{per}}^N } \ \left\{ \tilde{f} \in \mathfrak{U} \  | \  \left(\widehat{\lambda}(\tilde{x}_1) , \dots, \widehat{\lambda}(\tilde{x}_N) \right) \ \text{is} \ \mathbb{Q}-\text{independant} \right\} $$

Fix once and for all a sequence of distinct periodic orbits $(y_N)_{N \geq 1} \in \Omega_{\text{per}}$. Then:

$$ \left\{ f \in \mathfrak{U} \ | \  \dim_\mathbb{Q} \text{Vect}_\mathbb{Q} \widehat{\lambda}\left( \tilde{\Omega}_{\text{per}} \right) = \infty \right\} \supset \bigcap_{N \geq 1}  \left\{ \tilde{f} \in \mathfrak{U} \  | \  \left(\widehat{\lambda}(\tilde{y}_1) , \dots, \widehat{\lambda}(\tilde{y}_N) \right) \ \text{is} \ \mathbb{Q}-\text{independant} \right\}$$
$$ = \bigcap_{N \geq 1} \bigcap_{(m_1 , \dots, m_N) \in \mathbb{Z}^N \setminus \{ 0 \}} \left\{ \tilde{f} \in \mathfrak{U} \ | \ \sum_{i=1}^N m_i \widehat{\lambda}(\tilde{y}_i) \neq 0 \right\} $$
And this is a countable intersection of dense open sets. Indeed, fix some $N \geq 1$ and some integers $(m_1, \dots, m_N) \in \mathbb{Z}^N \setminus\{0\}$, and denote $ \mathfrak{U}_{m_1,\dots,m_N} := \left\{ \tilde{f} \in \mathfrak{U} \ | \ \sum_{i=1}^N m_i \widehat{\lambda}(\tilde{y}_i) \neq 0 \right\} $. Without loss of generality, we may suppose that $m_1 \neq 0$.\\

First of all, $\mathfrak{U}_{m_1,\dots,m_N}$ is open. This is because the function 
$$ \begin{array}[t]{lrcl}
  & \mathfrak{U}  & \longrightarrow & \mathbb{R}^N \\
    & \tilde{f} & \longmapsto &  (\widehat{\lambda}(\tilde{y}_1),\dots, \widehat{\lambda}(\tilde{y}_N))  \end{array} $$
is continuous. Indeed, recall that $$\widehat{\lambda}(\tilde{y}_i) = \frac{1}{n_i} \sum_{k=0}^{n_i} \log | \partial_u \tilde{f} (\tilde{f}^k( \tilde{y}_i )) |, $$ where $n_i$ is the period of $y_i \in \Omega_{\text{per}}$ so is constant in $\tilde{f}$. Moreover, $\tilde{f}$ varies smoothly in $C^1$ norm, $\tilde{y}_i$ varies continuously, and finally the unstable direction of $\tilde{f}$ also varies continuously in $\tilde{f}$, see  Corollary 2.9 in \cite{CP15}. \\

Now, we check that $\mathfrak{U}_{m_1,\dots, m_N}$ is dense in $\mathfrak{U}$. Without loss of generality, it suffices to prove that $f \in \overline{\mathfrak{U}_{m_1, \dots, m_N}}$. If $f \in \mathfrak{U}_{m_1,\dots, m_N}$ then we have nothing to prove. So suppose $f \notin \mathfrak{U}_{m_1,\dots, m_N}$. We are going to construct a perturbation $\tilde{f}$ in $\mathfrak{U}_{m_1,\dots,m_N}$. \\

Notice that the set $$ \Lambda := \left\{ f^k(y_i) \ | \ k \in \mathbb{Z}, i \in \llbracket 1 , N \rrbracket \right\} \subset {\Omega}_{\text{per}} $$
is discrete and finite. So there exists a small open neighborhood $U$ of $y_1$ such that $U \cap \Lambda = \{y_1\}$. Then, choose $\varphi:M \rightarrow M$ a diffeomorphism of $M$ that is a small perturbation of the identity on $M$ such that $\varphi=Id$ on $M\setminus U$, $\varphi(y_1)=y_1$, $(d\varphi)_{y_1}(E^u(y_1)) = E^u(y_1) $ and such that $\partial_u \varphi(y_1) = 1+\varepsilon$ for some small $\varepsilon>0$. \\

Define $\tilde{f} := f \circ \varphi$. Then $\tilde{f}$ is a small perturbation of $f$. Moreover, the $y_i$ are periodic orbits for $\tilde{f}$. Since $h$ sends periodic orbits to periodic orbits of the same period, and since $h$ is a small perturbation of the identity, it follows necessarily that $\tilde{y_i} = h(y_i) = y_i$. Hence, the Lyapunov exponents of interests are:

$$ \widehat{\lambda}(\tilde{y}_1)  = \frac{1}{n_1} \log( | \partial_u f(y_1) (1+\varepsilon) | ) + \frac{1}{n_1} \sum_{k=1}^{n_1-1} \log |\partial_u \tilde{f} (f^k(y_1)) |  = \frac{1}{n_1} \log(1+\varepsilon) + \widehat{\lambda}(y_1) .$$
$$ \widehat{\lambda}(\tilde{y}_i) = \widehat{\lambda}(y_i), \forall i \in \llbracket 2,N \rrbracket. $$

Hence $$ \sum_{i=1}^n m_i  \widehat{\lambda}(\tilde{y}_i) = \frac{1}{n_1} \log(1+\varepsilon) + \sum_{i=1}^{N} m_i  \widehat{\lambda}(y_i) = \frac{m_1}{n_1} \log(1+\varepsilon) \neq 0, $$
since $f$ is supposed to be $\mathfrak{U} \setminus \mathfrak{U}_{n_1,\dots, n_N}$ and $m_1 \neq 0$. The proof is done. \end{proof}

\section{An explicit solenoid satisfying Theorem 1.2}

\subsection{A nonlinear perturbation of the doubling map}

Our goal is to construct a nonlinear perturbation of the doubling map on the circle on which we will be able to compute some periodic orbits and Lyapunov exponents. We want them to be linearly independant over $\mathbb{Q}$.

\begin{lemma}
Let $\Lambda := \{ \frac{1}{2^K - 1} \ | \ K \geq 2 \}$ and $$ \tilde{\Lambda} := \bigsqcup_{K \geq 2 } \left[ \frac{1}{2^K -1} - \frac{1}{8^K} \ , \ \frac{1}{2^K -1} + \frac{1}{8^K} \right] $$

Let $N \geq 2$ and $k \in \llbracket 1, N-1 \rrbracket$. Then $$ \frac{2^k}{2^N-1} \notin \tilde{\Lambda} .$$

\end{lemma}

\begin{proof}

First of all, the union is indeed disjoint. Indeed, for $K \geq 2$, we see that 
$$ \frac{1}{2^{K+1}-1} + \frac{1}{8^{K+1}} < \frac{1}{2^{K}-1} - \frac{1}{8^{K}} $$
Since $$ \frac{1}{2^{K}-1} - \frac{1}{2^{K+1}-1} >  \frac{1}{2^{K}} - \frac{1}{\frac{3}{2} 2^{K}} = \frac{1}{3} \frac{1}{2^K} > \frac{9}{8} \frac{1}{8^K} .$$

Now, fix $N \geq 2$. We first do the case where $k=N-1$. 
In this case, $$ \frac{2^k}{2^N-1} = \frac{1}{2} \frac{1}{1-2^{-N}} > 1/2 ,$$
which proves that $ \frac{2^k}{2^N-1} \notin \tilde{\Lambda} $ since $ \tilde{\Lambda} \subset [0,0.4] $. 

We still have to do the case where $ k \in \llbracket 1, N-2 \rrbracket$, and $N \geq 3$. First of all, we check that
$$  \frac{2^k}{2^N-1} < \frac{1}{2^{N-k}-1} - \frac{1}{8^{N-k}} .$$

Indeed:
$$ \frac{1}{2^{N-k}-1} - \frac{2^k}{2^{N}-1} = \frac{1}{2^{N-k}-1} - \frac{1}{2^{N-k}-2^{-k}} $$
$$ \geq  \frac{1}{2^{N-k}-1} - \frac{1}{2^{N-k}-2^{-1}} $$
$$ = \frac{1}{2 (2^{N-k}-1)(2^{N-k}-2^{-1})}  $$
$$ > \frac{1}{2 \cdot 4^{N-k}} > \frac{1}{8^{N-k}}.$$

Second, we check that $$ \frac{1}{2^{N-k+1}-1} + \frac{1}{8^{N-k+1}} < \frac{2^k}{2^N - 1} . $$

Indeed, 
$$ \frac{2^k}{2^N - 1} - \frac{1}{2^{N-k+1}-1} = \frac{1}{2^{N-k} - 2^{-k}} - \frac{1}{2^{N-k+1}-1} $$
$$ > \frac{1}{2^{N-k}} - \frac{1}{\frac{3}{2} \cdot 2^{N-k}} $$
$$ = \frac{1}{3} \frac{1}{2^{N-k}} > \frac{1}{8^{N-k+1}} $$

And this conclude the proof.

\end{proof}

\begin{lemma}

Fix $(\alpha_N)_{N \geq 2}$ a family of real numbers so that $ \sum_{K \geq 2} |\alpha_K| 8^{K r} < \infty $ for any $r \geq 1$. There exists a $\mathcal{C}^\infty$ function $g : \mathbb{R} \rightarrow \mathbb{R}$ such that:

\begin{itemize}
    \item $ \text{supp g} \subset \tilde{\Lambda} $,
    \item $ \| g \|_{C^1} < 1 $.
    \item $ \forall N \geq 2, \ \forall k \in \llbracket 1, N-1 \rrbracket, \ g\left( \frac{2^k}{2^N-1} \right) = 0$
    \item $ \forall N \geq 2, \ g'\left( \frac{1}{2^N-1} \right) = \alpha_N$.
\end{itemize}

\end{lemma}

\begin{proof}

Let $\theta : \mathbb{R} \rightarrow [0,1]$ be a smooth bump function such that $\text{supp}(\theta) \subset [-1/2,1/2] $ and such that $\theta = 1$ on $[-1/4,1/4]$. Let $\chi(x) := x \theta(x)$. \\

The map $\chi$ is supported in $[-1/2,1/2]$ and satisfy $\chi(0)=0$, $\chi'(0)=1$.
Define, for $N \geq 2$, $$ \chi_N(x) := \chi\left(8^N\left(x-\frac{1}{2^N-1}\right)\right).$$
Then $\chi_N$ is supported in $\left[ \frac{1}{2^N -1} - \frac{1}{8^N} \ , \ \frac{1}{2^N -1} + \frac{1}{8^N} \right]$, and so the function
$$ g(x) := \sum_{K \geq 2} \alpha_K {8^{-K}} {\chi_K(x)} $$
is well defined and supported in $\tilde{\Lambda}$. Since, for all $r \geq 1$, $ \sum_{K \geq 2} |\alpha_K| 8^{K r} < \infty $, $g$ is a $\mathcal{C}^\infty$ function, with derivatives $$ g^{(r)}(x) = \sum_{K \geq 2} \alpha_K 8^{(r-1)K} \chi^{(r)}\left(8^N\left(x-\frac{1}{2^N-1}\right)\right) .$$
In particular, $$ \|g\|_{C^r} \leq \sum_{K \geq 2} |\alpha_K| 8^{(r-1)K} $$
And so we see that replacing the sequence $(\alpha_K)_{K \geq 2}$ by $(\alpha_{K+K_0})_{K \geq 2}$ allows us to choose $\|g\|_{C^r}$ as small as we want if so desired. The fact that $g$ vanishes on $\Lambda$ and the computation of its derivative on this set follows from the previous formulae. \end{proof}

\begin{lemma}

The map $f(x) := 2x + g(x)$ can be seen as a smooth map $\mathbb{S}^1 \longrightarrow \mathbb{S}^1$.

\end{lemma}

\begin{proof}

We will identify $[0,1]/\{0 \sim 1\}$ to the circle. The only thing to check is if $g$ stays smooth after the quotient. It stays smooth indeed, since $g^{(r)}(0)=g^{(r)}(1)=0$ for all $r \geq 0$. \end{proof}

\begin{lemma}

For $g$ small enough, $(f,\mathbb{S}^1)$ is a nonlinear perturbation of the usual doubling map. A familly of periodic points for this dynamical system is the set $ \Lambda $, seen as a subset of the circle. The associated lyapunov exponents are  $\ln(2) + \frac{1}{N}\ln\left(1+\alpha_N / 2 \right)$.

\end{lemma}

\begin{proof}

Let $N \geq 2$ and let $x_N := \frac{1}{2^N -1}$. We check that $x_N$ is periodic with period $N$. Indeed, for all $k \in \llbracket 1, N-1 \rrbracket$, we see that $ f^k(x_N) = 2^k x_N $ by construction of $g$, and then
$$ f^N(x_N) = \frac{2^N}{2^N -  1} = \frac{1}{2^N - 1} \ \text{mod} \ 1 .$$
Hence $x_N$ is $f$-periodic with period $N$. Its lyapunov exponent is then
$$ \widehat{\lambda}( x_N ) := \frac{1}{N} \sum_{k=0}^{N-1} \ln |f'(f^k(x_N))| 
= \frac{1}{N} \sum_{k=0}^{N-1} \ln |2+g'(2^k x_N))| = \ln(2) + \frac{1}{N}\ln\left(1+\alpha_N / 2 \right) .  $$  \end{proof}

\begin{lemma}

We can choose the sequence $(\alpha_N)$ so that the Lyapunov exponents $\widehat{\lambda}( x_N )$ is a family a real numbers that are linearly independent over $\mathbb{Q}$.

\end{lemma}

\begin{proof}

By the {Lindemann–Weierstrass theorem} \cite{Ba90}, we just have to choose the $\alpha_N$ so that
$$ \widehat{\lambda}(x_N) = e^{\beta_N} $$
where $\beta_N$ are distinct algebraic numbers (while still ensuring that the sums of Lemma B.2 converges).

In other words, it suffice to choose $\alpha_N$ of the form
$$ \alpha_N = 2 \left( 2^{-N} e^{N \exp(\beta_N) }  - 1\right) $$

with $\beta_N$ distinct algebraic numbers converging to $\ln \ln 2$ quickly enough.
For example, we can fix:
$$ \beta_N := \left\lfloor 10^{N^2} \ln \ln 2 \right\rfloor 10^{- N^2} $$
And in this case $\alpha_N = O(N \cdot 10^{-N^2})$, which is quick enough. \end{proof}

We have constructed a chaotic and \emph{nonlinear} dynamical system on the circle with some prescribed Lyapunov exponent. With our choice of $\alpha_N$, we have $$ \widehat{\lambda}(x_N) = e^{ \left\lfloor 10^{N^2} \ln \ln 2 \right\rfloor 10^{- N^2} }.$$
In particular, $$ \text{dim}_\mathbb{Q} \text{Vect}_\mathbb{Q} \left\{ \widehat{\lambda}(x) \ | \ x \text{ is periodic} \right\} = \infty$$
which was what we wanted to construct.

\subsection{A nonlinear solenoid}

In this section we construct an explicit nonlinear perturbation of the usual solenoid. Denote by $\mathbb{T} := \mathbb{R}/\mathbb{Z} \times \overline{\mathbb{D}} $ the full torus.
Define $ F : \mathbb{T} \rightarrow \mathbb{T} $ by the formula

$$ F(\theta,x,y) := \left(f(\theta), \frac{1}{4}x + \frac{1}{4 \pi} \cos(2 \pi \theta) ,  \frac{1}{4}y + \frac{1}{4 \pi} \sin(2 \pi \theta)  \right)  $$

where $f(\theta) = 2 \theta + g(\theta)$ is the function defined previously. Notice that $F$ is a diffeomorphism onto its image. In particular, we can use it to glue $\mathbb{T}$ with itself along its boundary, which allows us to see $F$ as a diffeomorphism of a genuine closed 3-manifold (\cite{Bo78}, chapter 1) which countains $\mathbb{T}$. \\

Then, $F(\mathbb{T}) \subset \overset{\circ}{\mathbb{T}} $, and so $T^n(\mathbb{T})$ is a strictly decreasing sequence of compact sets. The intersection $ S := \bigcap_{n \geq 0} F^n(\mathbb{T}) $ is called a (nonlinear) solenoid. The set $S$ is an attractor for $F$. Since $F$ is a smooth perturbation of the usual solenoid, which is a structurally stable Axiom A diffeomorphism (this follows for example from Theorem 1 in \cite{IPR10}), our dynamical system is still Axiom A. Moreover, it has codimension 1 stable lamination, as it contract in the $(x,y)$ variables. \\

Indeed, at a given point $p=(\theta,x,y) \in S$, we see that the Jacobian of $F$ is
$$ \text{Jac}_F(p) = \begin{pmatrix} 2+g'(\theta) & 0 & 0 \\ - \frac{1}{2}\sin(2 \pi \theta) & 1/4 & 0 \\ \frac{1}{2}\cos(2 \pi \theta) & 0 & 1/4 \end{pmatrix} $$

In particular, the subspace $E^s(p) := \{(0,h,k) \ | \ h,k \in \mathbb{R} \} \subset T_p \mathbb{T} = \mathbb{R}^3$ is independant of $p$, and we check that $(dF)_p (E^s(p)) \subset E^s(F(p))$ is a contracting linear map. So we have found the stable direction. \\

This allows us to compute the derivative in the unstable direction at $p$. Indeed, since $$ (dF)_p : E^s(p) \oplus E^u(p) \rightarrow E^s(F(p)) \oplus E^u(F(p)) $$
also sends $E^u(p)$ into $E^u(F(p))$, we can compute the determinant of the jacobian by making the unstable derivative appear like so:

$$ \det (dF)_p = \partial_u F(p) \times \det( (dF)_{\ |E^s} ) $$

And we already know that $$ \det( (dF)_{\ |E^s} ) = \det(I/4) = 1/16 $$
and that $$ \det(dF)_p = (2+g'(\theta))/16 .$$

Hence:
$$ \partial_u F(p) = f'(\theta) .$$

Notice that the bunching condition (B) is satisfied. \\

Finally, we are going to construct some periodic orbits for $F$ and compute their unstable Lyapunov exponents. If $p \in S$ is a periodic point for $F$, then it is clear that its angular coordinate is periodic for $f$. Reciprocally, let $\theta_0$ be a periodic point for $f$: there exists a integer $n_0$ such that $f^{n_0}(\theta_0)=\theta_0$.
Then, the map $ F^{n_0} $ satisfy
$$ F^{n_0}( \{  \theta_0 \} \times \overline{\mathbb{D}} ) \subset \{ \theta_0 \} \times \mathbb{D} .$$
and is contracting on this disk. Hence, there exists a unique associated fixed point $p_0 \in \{ \theta_0\} \times \mathbb{D}$. \\

This allows us to exhibit some periodic orbits and compute the associated Lyapunov exponents. Let $N \geq 2$ and consider $\theta_N := \frac{1}{2^N-1}$ the periodic point for $f$ constructed in the last subsection. Let $p_N$ be the unique associated fixed point, noted $$ p_N = (\theta_N, x_N, y_N ) \in S .$$
Then we have proved that its associated unstable Lyapunov exponent is

$$ \widehat{\lambda}(p_N) = \widehat{\lambda}(\theta_N) = e^{\beta_N} .$$

In particular, we have
$$ \dim_\mathbb{Q} \text{Vect}_\mathbb{Q} \widehat{\lambda} \left( S_\text{per} \right) = \infty $$

where $S_{\text{per}} := \{ p \in S \ | \ p \ \text{is} \ F \ \text{periodic}. \}$. Hence, Theorem 1.2 applies. For example, the SRB measure and the measure of maximal entropy enjoys polynomial Fourier decay in the unstable direction. Indeed, let $\chi$ be a Hölder map with localized support at a point $p_0$ on the solenoid. Let $\mu$ be an equilibrium state. Then, the Fourier transform of $ \chi d \mu $ write, for $\xi \in \mathbb{R}^3$:
$$ \widehat{\chi d \mu}(\xi) = \int_{S} e^{- 2 i \pi x \cdot \xi} \chi(x) d\mu(x) = \int_{S} e^{- i t \phi_v} \chi d\mu $$

where $v := \xi/|\xi|$, $t=2 \pi |\xi|$ and $ \phi_v(x):= x \cdot v$.
Fix a direction $v$ such that $v$ is not orthogonal to $E^u(p_0)$. Then, if $\chi$ has small enough support, $ |\partial_u \phi_v| > 0 $ around the support of $\chi$. It follows from Theorem 1.2 that  $\widehat{\chi d \mu}(t v) \underset{t \rightarrow \infty}{\longrightarrow} 0$ at a polynomial rate. By Remark 4.1, the convergence to zero is even uniform on any cone in the unstable direction.

\end{appendices}

\end{document}